\documentclass{amsart}

\newtheorem{theorem}{Theorem}[section]
\newtheorem{lemma}[theorem]{Lemma}
\newtheorem{proposition}[theorem]{Proposition}

\theoremstyle{definition}

\newtheorem{notation}[theorem]{Notation}

\newtheorem{conjecture}[theorem]{Conjecture}

\theoremstyle{remark}
\newtheorem{remark}[theorem]{Remark}

\numberwithin{equation}{section}

\begin{document}

\title[Algebraic Montgomery-Yang Problem]{Algebraic Montgomery-Yang Problem: the non-rational case and the del Pezzo case}

\author{DongSeon Hwang}
\address{School of Mathematics, Korea Institute For Advanced Study, Seoul 130-722, Korea}
\email{dshwang@kias.re.kr}

\thanks{Research supported by Basic Science Research Program through the National Research Foundation(NRF)
of Korea funded by the Ministry of education, Science and
Technology (NRF-2007-C00002).}

\author{JongHae Keum}
\address{School of Mathematics, Korea Institute For Advanced Study, Seoul 130-722, Korea}
\email{jhkeum@kias.re.kr}


\date{July 23, 2010}



\begin{abstract}
Montgomery-Yang problem predicts that every pseudofree circle
action on the 5-dimensional sphere has at most $3$ non-free
orbits. Using a certain one-to-one correspondence, Koll\'ar
formulated the algebraic version of the Montgomery-Yang problem:
every projective surface $S$ with the second Betti number $b_2(S)
= 1$ and with quotient singularities has at most $3$ singular
points if its smooth locus $S^0$ is simply-connected. In a
previous paper, we have confirmed the conjecture when $S$ has at
least one non-cyclic quotient singularity. In this paper, we prove
the conjecture either when $S$ is not rational or when $-K_S$ is
ample. Thus the conjecture is reduced to the case where $S$ is a
rational surface with $K_S$ ample having at worst cyclic
singularities.
\end{abstract}

\maketitle



\section{Introduction}

A pseudofree $\mathbb{S}^1$-action on a sphere $\mathbb{S}^{2k-1}$
is a smooth $\mathbb{S}^1$-action which is free except for
finitely many non-free orbits (whose isotropy types
$\mathbb{Z}_{m_1}, \ldots, \mathbb{Z}_{m_n}$ have pairwise
relatively prime orders).

For $k=2$ Seifert \cite{Sei} showed that such an action must be
linear and hence has at most two non-free orbits. In the contrast
to this, for $k=4$  Montgomery and Yang \cite{MY} showed that
given any pairwise relatively prime collection of positive
integers $m_1,\ldots, m_n$, there is a pseudofree
$\mathbb{S}^1$-action on homotopy 7-sphere whose non-free orbits
have exactly those orders. Petrie \cite{Pet} proved similar
results in all higher odd dimensions. This led Fintushel and Stern
to formulate the following problem:

\begin{conjecture} [\cite{FS87}](Montgomery-Yang Problem) \\
{\it Let $$\mathbb{S}^1 \times \mathbb{S}^5 \rightarrow
\mathbb{S}^5$$ be a pseudo-free $\mathbb{S}^1$-action.
 Then it has at most $3$ non-free orbits.}
\end{conjecture}

The problem has remained unsolved since its formulation.

Pseudofree $\mathbb{S}^1$-actions on 5-manifolds $L$ have been
studied in terms of the 4-dimensional quotient orbifold
$L/\mathbb{S}^1$ (see e.g., \cite{FS85}, \cite{FS87}). The
following one-to-one correspondence was known to Montgomery, Yang,
Fintushel and Stern, and recently observed by Koll\'ar
(\cite{Kol05}, \cite{Kol08}):

\begin{theorem} [cf. \cite{Kol05}, \cite{Kol08}]
There is a one-to-one correspondence between:
\begin{enumerate}
\item Pseudofree $\mathbb{S}^1$-actions on $5$ dimensional
rational homology spheres $L$ with $H_1(L, \mathbb{Z}) = 0$.
\item Smooth, compact $4$ manifolds $M$ with boundary such that
\begin{enumerate}
\item $\partial{M} = \cup_i L_i$ is a disjoint union of lens spaces $L_i = \mathbb{S}^3/\mathbb{Z}_{m_i}$,
\item the $m_i$ are relatively prime to each other,
\item $H_1(M, \mathbb{Z}) = 0$ and $H_2(M, \mathbb{Z}) \cong \mathbb{Z}$.
\end{enumerate}
\end{enumerate}
Furthermore, $L$ is diffeomorphic to $\mathbb{S}^5$ iff $\pi_1(M)
= 1$.
\end{theorem}

We recall that a normal projective surface with the same Betti
numbers with the projective plane $\mathbb{C}\mathbb{P}^2$ is
called a \emph{rational homology projective plane}, a
\emph{$\mathbb{Q}$-homology projective plane} or a
\emph{$\mathbb{Q}$-homology $\mathbb{C}\mathbb{P}^2$}. When a
normal projective surface $S$ has quotient singularities only, $S$
is a $\mathbb{Q}$-homology projective plane if the second Betti
number $b_2(S)=1$.

It is known that a $\mathbb{Q}$-homology projective plane with
quotient singularities has at most 5 singular points (cf.
\cite{HK1} Corollary 3.4). Recently, the authors have classified
$\mathbb{Q}$-homology projective planes with 5 quotient
singularities (\cite{HK1}, also see \cite{K10}).

Using the one-to-one correspondence, Koll\'ar formulated the
algebraic version of the Montgomery-Yang problem as follows:

\begin{conjecture} [\cite{Kol08}] (Algebraic Montgomery-Yang Problem)\\
{\it Let $S$ be a $\mathbb{Q}$-homology projective plane with
quotient singularities. Assume that $S^0:=S \backslash Sing(S)$ is
simply-connected. Then $S$ has at most $3$ singular points.}
\end{conjecture}

In a previous paper \cite{HK2}, we have confirmed the conjecture
when $S$ has at least one non-cyclic quotient singularity.

In this paper, we consider the case where $S$ has cyclic
singularities only. We first verify the conjecture when $S$ is not
rational.

\begin{theorem}\label{main}
Let $S$ be a $\mathbb{Q}$-homology projective plane with cyclic
singularities. Assume that $H_1(S^0, \mathbb{Z}) = 0$. If $S$ is
not rational, then $S$ has at most $3$ singular points.
\end{theorem}

\begin{remark}
The condition $H_1(S^0, \mathbb{Z}) = 0$ is weaker than the
condition $\pi(S^0) =\{1\}$, and there are examples of
$\mathbb{Q}$-homology projective planes with $4$ quotient
singularities, not all cyclic, such that $H_1(S^0, \mathbb{Z}) =
0$. Such surfaces are completely classified in \cite{HK2}. It
turns out that they are log del Pezzo surfaces with $3$ cyclic
singularities and 1 non-cyclic singularity such that $H_1(S^0,
\mathbb{Z}) = 0$ but $\pi_1(S^0) \cong \mathfrak{A}_5$, the simple
group of order 60.
\end{remark}

Next, we also prove the conjecture when $-K_S$ is ample.

\begin{theorem}\label{dpmain}
Let $S$ be a $\mathbb{Q}$-homology projective plane with cyclic
singularities. Assume that $H_1(S^0, \mathbb{Z}) = 0$. If $-K_S$
is ample, then $S$ has at most $3$ singular points.
\end{theorem}

\begin{remark}
(1) The condition $H_1(S^0, \mathbb{Z}) = 0$ implies that $K_S$ is
not numerically trivial, i.e., $K_S$ or  $-K_S$ is ample (Lemma
\ref{coprime}). Thus, Theorems \ref{main} and \ref{dpmain}
together reduce Conjecture 1.3 to the case where $S$ is a rational
surface with cyclic singularities such that $K_S$ is ample.

(2) Rational surfaces $S$ with cyclic singularities have been
studied extensively when $-K_S$ is ample or numerically trivial.
In the former case the surface is called a log del Pezzo surface,
and in the latter the surface is called a log Enriques surface. On
the other hand, when $K_S$ is ample, very little is known about
the classification of such surfaces. Moreover, if in addition
$b_2(S)=1$, that is, if $S$ is a $\mathbb{Q}$-homology projective
plane  with $K_S$ ample having at worst cyclic singularities,
nothing seems to be known except the examples due to Koll\'ar
(\cite{Kol08}, Example 43). He constructed a series of such
examples by contracting two rational curves on some well-chosen
weighted projective hypersurfaces. Koll\'ar's examples have
$|Sing(S)|=2$. In \cite{HK3} we give new examples with
$|Sing(S)|=1$, 2, or 3, all constructed geometrically, i.e., by
blowing up the projective plane and then contracting chains of
rational curves.
\end{remark}

The proof of Theorem \ref{main} goes as follows.

Let $S$ be a $\mathbb{Q}$-homology projective plane with cyclic
singularities such that $H_1(S^0, \mathbb{Z}) = 0$. Then the
orders of local fundamental groups of singular points are pairwise
relatively prime (Lemma \ref{coprime}). Also, by the orbifold
Bogomolov-Miyaoka-Yau inequality (see Theorems \ref{bmy},
\ref{bmy2}) $S$ has at most $4$ singular points. Assume that $S$
has $4$ singular points. Then the same inequality enables us to
enumerate all possible $4$-tuples consisting of the orders of
local fundamental groups of singular points:

\bigskip
\begin{displaymath}
\begin{array}{llllll}
(2, 3, 5, q),& q \geq 7,&\gcd(q,30)=1;  &\\
(2, 3, 7, q),& 11\le q \leq 41,&\gcd(q,42)=1;  &\\
(2,3,11,13).& &&\\
\end{array}
\end{displaymath}
\bigskip

Given its minimal resolution $$f : S' \rightarrow S,$$ the
exceptional curves and the canonical class $K_{S'}$ span a
sublattice $$R+\langle K_{S'}\rangle$$ of the unimodular lattice
$$H^2(S', \mathbb{Z})_{free}:=H^2(S', \mathbb{Z}){\rm /(torsion)},$$ where
$R$ is the sublattice spanned by the exceptional curves. We note
that $K_S$ is not numerically trivial (Lemma \ref{coprime}), hence
$R+\langle K_{S'}\rangle$ is of finite index in $H^2(S',
\mathbb{Z})_{free}$. As a consequence, its discriminant
$$D:=|\det(R+\langle K_{S'}\rangle)|$$ is a positive square number
(Lemma \ref{coprime}). This criterion significantly reduces the
infinite list of all possible cases for $R$. For example, the
order 3 singularity of the case $(2, 3, 5, q)$ must be of type
$\frac{1}{3}(1,1)$ (Lemma \ref{noA2}). The reduced list is still
infinite, and almost all cases in the list cannot be ruled out by
any further argument from lattice theory, e.g. computation of
$\epsilon$-invariants does not work here, which turned out to be
effective in the proof of \cite{HK1}. To handle this infinite
list, we compute $(-1)$-curves on the minimal resolution $S'$.
Assume further that $S$ is not rational. This assumption implies
that $K_S$ is ample and $S'$ contains a $(-1)$-curve $E$ with
$E.(f^*K_S/K_S^2)$ small, i.e., with $(f^*K_S/K_S^2)$-degree small
(Lemma \ref{bound}). Then we proceed to prove that the existence
of such a $(-1)$-curve $E$ leads to a contradiction by using
certain expressions of the intersection numbers $EK_{S'}$ and
$E^2$ in terms of the intersection numbers of $E$ with the
exceptional curves and $f^*K_{S}$ (Proposition \ref{int}). Here we
also use the classification result for the case of 5 singular
points \cite{HK1}.

The idea of computing $(-1)$-curves on the minimal resolution was
first used in \cite{K08} for some fixed types of singularities. In
Proposition \ref{int}, we derive general formulas for arbitrary
cyclic singularities. These formulas are useful in proving the
non-existence of a curve on $S'$ with prescribed intersection
numbers with the exceptional curves.

The proof of Theorem \ref{dpmain} is given in Section 7 and 8.
Here we also need, besides the previous ingredients, some detailed
properties of del Pezzo surfaces of rank one with cyclic
singularities developed by Zhang \cite{Zhang}, Gurjar and Zhang
\cite{GZ} and Belousov \cite{Belousov}.

\medskip
 Throughout this paper, we work over the field $\mathbb{C}$ of complex numbers.

\bigskip {\bf Notation}

\bigskip\noindent
$\bullet$ $[n_1, n_2, \ldots, n_l]$ a Hirzebruch-Jung continued
fraction, i.e.,

 \[
[n_1, n_2, ..., n_l]= n_1 - \dfrac{1}{n_2-\dfrac{1}{\ddots -
\dfrac{1}{n_l}}}= \frac{q}{q_1}
 \]

\medskip corresponding to a cyclic singularity of type
$\frac{1}{q}(1, q_1)$.\\
$\bullet$ $|[n_1, n_2, \ldots, n_l]|=q$.\\
$\bullet$ $b_i(X)$ the $i$-th Betti number of a complex variety
$X$.\\$\bullet$ $f:S' \rightarrow S$ a minimal resolution of a
normal surface $S$.\\$\bullet$ $Sing(S)$: the singular locus of
$S$.
\\$\bullet$ $\mathcal{F}:=f^{-1}(Sing(S))$ a reduced integral
divisor on $S'$.
\\$\bullet$ $R_p$:
the sublattice of $H^2(S', \mathbb{Z})_{free}$ spanned by the
numerical classes of the components of $f^{-1}(p)$, where $H^2(S',
\mathbb{Z})_{free}=H^2(S', \mathbb{Z})/({\rm torsion})$.
\\$\bullet$ $R:= \oplus_{p \in Sing(S)} R_p$  the sublattice of $H^2(S', \mathbb{Z})_{free}$ spanned by the
numerical classes of the irreducible exceptional curves of
$f:S'\to S$.
\\$\bullet$ $L=L_S:={\rm rank}(R)$, the number of the irreducible
components of $\mathcal{F}=f^{-1}(Sing(S))$, or the number of the
exceptional curves of $f:S' \rightarrow S$.

\section{Hirzebruch-Jung Continued Fractions}
 Let $\mathcal{H}$ be the set of all Hirzebruch-Jung continued
fractions $ [n_1, n_2, \ldots, n_l]$,
$$\mathcal{H} = \underset{l \geq 1}{\bigcup} \{ [n_1, n_2, \ldots, n_l] \mid \textrm{all}\,\,
n_j\,\, \textrm{are integers} \geq 2 \}.$$

\begin{notation} Fix  $w = [n_1, n_2,
\ldots, n_l] \in \mathcal{H}$.
    \begin{enumerate}
        \item  The $length$ of $w$, denoted by $l(w)$, is the number of entries of $w$.
        \item The $trace$ of $w$, $tr(w)=\overset{l}{\underset{j = 1}{\sum}} n_j$, is the sum of
        entries of $w$.
        \item $|w|=|[n_1, n_2, \ldots, n_l]|:= |\det(M(-n_1, \ldots,
        -n_l))|,$
 where
\begin{displaymath}
M(-n_1, \ldots, -n_l) = \left( \begin{array}{cccccc}
-n_1 & 1 & 0 & \cdots & \cdots &0 \\
1 & -n_2 & 1 & \cdots & \cdots& 0 \\
0 & 1 & -n_3 & \cdots & \cdots& 0 \\
\vdots & \vdots & \vdots & \ddots & \vdots & \vdots\\
0 & 0& 0 & \cdots  & -n_{l-1}&1  \\
0 & 0& 0 & \cdots &1 & -n_l \\
\end{array} \right)
\end{displaymath}
is the intersection matrix of $[n_1, n_2, \ldots, n_l]$.
\item $q:=|w|=$ the order of
the cyclic singularity corresponding to $w$, i.e.,
$w=\frac{q}{q_1}$ for some $q_1$ with $1\le q_1<q, \,\, \gcd(q,
q_1)=1$.
$$ q_{a_1, a_2, \ldots, a_m} := |\det(M')|,$$
$$q_{1,2, \ldots,
l}:=|\det(M(\emptyset))|=1,$$ where
$M'$ is the $(l-m)\times(l-m)$ matrix obtained by deleting\\
$-n_{a_1}, -n_{a_2}, \ldots, -n_{a_m}$ from $M(-n_1, \ldots,
-n_l)$.  For example, $$q_1 = |\det(M(-n_2, \ldots, -n_l))|=|[n_2,
n_3, \ldots, n_l]|$$
$$q_l = |\det(M(-n_1, \ldots, -n_{l-1}))|=|[n_1, n_2, \ldots, n_{l-1}]|$$
$$q_{1, l} = |\det(M(-n_2, \ldots, -n_{l-1}))|=|[n_2,n_3, \ldots, n_{l-1}]|.$$
Note that
$$[n_l, n_{l-1}, ..., n_1] = \dfrac{q}{q_l}$$
$$q_1 q_l = q_{1,l} q + 1\,\,\, {\rm if}\,\, l\ge 2.$$
\end{enumerate}
\end{notation}

We will write simply $l$, $tr$ for $l(w)$, $tr(w)$ if there is no
confusion.

\medskip
The following number-theoretic property of Hirzebruch-Jung
continued fractions will play a key role in the proof of Lemma
\ref{noA2}.

\begin{proposition}\label{det-trace} For  $w = [n_1, n_2,
\ldots, n_l] \in \mathcal{H}$, $$q_1+q_l+tr \cdot q \not \equiv 0
\,\,\,{\rm modulo}\,\,\, 3\,\,\, {\rm iff}\,\,\, q \equiv 0
\,\,\,{\rm modulo}\,\,\, 3.$$
\end{proposition}


\begin{proof} In the following, $a \equiv b$ means that $a\equiv b$ modulo $3$.

\medskip Assume $q\equiv 0.$\\
If $l=1$ and $w=[n_1]$, then $q_1=q_l=|\det(M(\emptyset))|=1$ and
$q=tr=n_1\equiv 0$, hence $$q_1+q_l+tr \cdot q\equiv 1+1+0 \not
\equiv 0.$$ If $l\ge 2$, then we see from the equality $q_1 q_l =
q_{1,l} q + 1$ that $q_1 q_l \equiv 1$. Thus $q_1\equiv
q_l\equiv\pm 1$ and
$$q_1+q_l+tr \cdot q\equiv \pm 1\pm 1 +0 \not \equiv 0.$$

\medskip
Assume $q\not \equiv 0$, i.e.,  $q\equiv \pm 1.$\\
We will show by induction on $l$ that \begin{equation}\label{tr}
q_1+q_l+tr \cdot q\equiv 0
\end{equation}
If $l=1$ and $w=[n_1]$, then $q_1=q_l=1$ and $q=tr=n_1\equiv \pm
1$, hence $$q_1+q_l+tr \cdot q\equiv 1+1+(\pm 1)^2 \equiv 0.$$ If
$l=2$ and $w=[n_1, n_2]$, then $q=n_1n_2-1\equiv \pm 1$, so
$n_1n_2\equiv -1$ or $0$, hence $n_1\equiv -n_2$ or $n_1\equiv 0$
or $n_2\equiv 0$. In any case,
$$q_1+q_l+tr \cdot q=n_2+n_1+(n_1+n_2)(n_1n_2-1)=n_1n_2(n_1+n_2)
\equiv 0.$$ Now assume $l\ge 3$. We divide the proof into 3 cases
$q_1\equiv 1, -1, 0$.

\medskip
Case (1): $q_1\equiv 1$. By the induction hypothesis \eqref{tr}
holds for $[n_2,\ldots,n_l]$, i.e.,
$$q_{1,2}+q_{1,l}+(tr-n_1)\cdot q_1\equiv 0.$$
Plugging $q=n_1q_1-q_{1,2}$ into the above equality, we get
$$q_{1,l}+tr\cdot q_1-q\equiv 0.$$
Thus $$\begin{array}{lll}q_1+q_l+tr\cdot q &\equiv& 1+q_l+tr\cdot
q\\ &\equiv& -1-1+1\cdot q_l+tr\cdot q\\ &\equiv&
-1-q^2+q_1q_l+tr\cdot q\\ &=& q_{1,l}q+tr\cdot q-q^2\\ &=&
(q_{1,l}+tr -q)q \\ &\equiv& (q_{1,l}+tr\cdot q_1-q)q\\
&\equiv& 0.\end{array}$$

\medskip
Case (2): $q_1\equiv
-1$. In this case, the induction hypothesis also gives\\
$q_{1,l}+tr\cdot q_1-q\equiv 0.$ Thus
$$\begin{array}{lll} q_1+q_l+tr\cdot q &\equiv& -1+q_l+tr\cdot
q\\ & \equiv& 1-q_1q_l+tr\cdot q+q^2\\ &\equiv& -q_{1,l}q-tr\cdot
q_1q+q^2\\ &=& -(q_{1,l}+tr\cdot q_1-q)q\\ &\equiv&
0.\end{array}$$

\medskip
Case (3): $q_1\equiv 0$. First note that $q=n_1q_1-q_{1,2}\equiv
-q_{1,2}$, so
$q_{1,2}\equiv -q\not \equiv 0$.\\
Also, note that $q_{1,l}q=q_1 q_l-1 \equiv -1$, so $q_{1,l}\equiv
-q$.\\ Since $q_{1,2}\not \equiv 0$, we apply the induction
hypothesis to $[n_3,\ldots,n_{l}]$ to get
$$q_{1,2,3}+q_{1,2,l} +(tr-n_1-n_2)\cdot
q_{1,2}\equiv 0.$$ Note that $q_1=n_2q_{1,2}-q_{1,2,3}$ and
$n_1q_{1,l}-q_l=q_{1,2,l}$.\\ Since $q_{1,2}\equiv q_{1,l}\equiv
-q$, we have
$$\begin{array}{lll} q_1+q_l+tr\cdot q &\equiv& q_1+q_l-tr\cdot
q_{1,2}\\ &\equiv&
q_1-(n_1q_{1,l}-q_l)-tr\cdot q_{1,2}+n_1q_{1,2}\\
&=& (n_2q_{1,2}-q_{1,2,3})-q_{1,2,l}-tr\cdot q_{1,2}+n_1q_{1,2}\\
& =& -q_{1,2,3}-q_{1,2,l} -(tr-n_1-n_2)\cdot q_{1,2}\\ &\equiv&
0.\end{array}$$
\end{proof}

We collect some properties of Hirzebruch-Jung continued fractions
which will be frequently used in the subsequent sections.

\begin{notation}\label{uv-n}
     For a fixed continued fraction $w = [n_1, n_2, \ldots,
        n_l] \in \mathcal{H}$ and an integer $0 \leq s \leq l+1$ , we define
        \begin{enumerate}
            \item   $u_s :=q_{s,\ldots, l}= |[n_1, n_2, \ldots, n_{s-1}]|\,\, (2\leq s\leq l+1), \quad u_0=0, \,\,u_1=1$
            \item   $v_s :=q_{1,\ldots, s}= |[n_{s+1}, n_{s+2}, \ldots,
            n_l]| \,\, (0\leq s\leq l-1), \quad v_l=1, \,\,v_{l+1}=0$.
        \end{enumerate}
        Note that $u_l=q_l$, $u_{l+1}=q$, $v_0=q$, $v_1=q_1$.
\end{notation}

\begin{lemma} \label{cf}
Let $w = [n_1, n_2, \ldots, n_l] \in \mathcal{H}$. Then,
    \begin{enumerate}
    \item $u_{j+1} = n_j  u_{j} - u_{j-1}$,\\
$v_{j-1} = n_j v_j -v_{j+1}$.
    \item $v_j  u_{j+1}- v_{j+1}u_{j} =v_{j-1}  u_{j} - v_j  u_{j-1}= q. $
    \item $v_j  u_{j} = \frac{1}{n_j}(q + v_{j+1}  u_{j} +v_j  u_{j-1}).$
    \item $\overset{s}{\underset{j = 1}{\sum}}  (n_j - 2) u_{j} = u_{s+1} - u_{s} - 1$,\\  $\overset{l}{\underset{j = s}{\sum}}  (n_j - 2) v_j = v_{s-1} - v_{s} -
    1$.
    \item $\frac{u_j+ v_j}{q}\le 1$.
    \item $|[n_1,\ldots, n_{j-1}, n_j+1, n_{j+1}, \ldots, n_l]|=u_jv_{j}+|[n_1, n_2,
\ldots, n_l]|>q.$
    \end{enumerate}
\end{lemma}

\begin{proof}
    (1) is well-known.\\
    (2) is obtained by a direct calculation using (1) as follows:
    \begin{eqnarray*}
     v_j  u_{j+1} - v_{j+1}  u_{j}
    & = & (n_j  u_{j}- u_{j-1} ) v_j - v_{j+1} u_{j}\\
    & = & (n_j v_j - v_{j+1})  u_{j} - v_j u_{j-1}\\
    & = & v_{j-1} u_{j} - v_j  u_{j-1}\\
    &   & \ldots\\
    & = & v_{1} u_{2} - v_2  u_{1}=q_1n_1-q_{1,2}=q.
    \end{eqnarray*}
    (3) follows from the equality $$n_jv_ju_j=(v_{j-1}+v_{j+1})u_j=q+v_j  u_{j-1} + v_{j+1}  u_{j}.$$
    (4) follows from $$\begin{array}{lll} (n_j - 2) u_{j} &=& (u_{j+1} - u_{j})-(u_{j}-
    u_{j-1})\\ (n_j - 2) v_{j} &=& (v_{j+1} - v_{j})-(v_{j}- v_{j-1}).\end{array}$$
    (5) Note that $$v_{j}= n_{j+1} v_{j+1}-v_{j+2}\ge  v_{j+1}+( v_{j+1}-v_{j+2})\ge  v_{j+1}+1, $$   $$u_{j+1}= n_j u_{j}-u_{j-1}\ge u_{j}+( u_{j}-u_{j-1})\ge u_{j}+1.$$ Thus
    $$q-(v_j+u_j)=v_j(u_{j+1}-1)-(v_{j+1}+1)u_{j}\ge 0.$$
    (6) Note that $$|[n_1,\ldots, n_{j-1}, n_j+1]|=(n_j+1)u_j-u_{j-1}= u_j+u_{j+1}.$$ By (2)
$$\begin{array}{lll} |[n_1,\ldots, n_{j-1}, n_j+1, n_{j+1}, \ldots, n_l]|&=&|[n_1,\ldots, n_{j-1}, n_j+1]|v_j-u_jv_{j+1}\\ &=&u_jv_{j} + u_{j+1}v_j-u_jv_{j+1}
\\ &=& u_jv_{j}+|[n_1, n_2, \ldots, n_l]|.\end{array}$$
\end{proof}

\begin{lemma} \label{uv} Assume $l\geq 5$. Then for arbitrary non-negative integers $z_1, \ldots,
    z_l$,
     \begin{enumerate}
     \item $\overset{l}{\underset{j = 1}{\sum}} (u_{j}+v_j)z_j \leq \overset{l}{\underset{j = 1}{\sum}}
    (u_{j}v_j)z_j^2$ when $\overset{l}{\underset{j = 1}{\sum}}
    z_j\geq 3$,\item
$\overset{l}{\underset{j = 1}{\sum}} (u_{j}+v_j)z_j \leq
\overset{l}{\underset{j = 1}{\sum}}
    (u_{j}v_j)z_j^2+2$ when  $\overset{l}{\underset{j = 1}{\sum}}
    z_j=2$,\item
    $\overset{l}{\underset{j = 1}{\sum}} (u_{j}+v_j)z_j\leq
\overset{l}{\underset{j = 1}{\sum}}
    (u_{j}v_j)z_j^2+1$ when  $\overset{l}{\underset{j = 1}{\sum}}
    z_j=1$.  \end{enumerate}
    \end{lemma}

\begin{proof} Note that $(u_{1}+v_1)z_1=(1+v_1)z_1\leq
v_1z_1^2-2$ if $z_1\geq 2$, and\\ $(u_{1}+v_1)z_1=(1+v_1)z_1=
v_1z_1^2+1$ if $z_1=1$.\\ Similarly,
$(u_{l}+v_l)z_l=(u_l+1)z_l\leq u_lz_1^2-2$ if $z_l\geq 2$, and
\\ $(u_{l}+v_l)z_l=(u_l+1)z_l= u_lz_l^2+1$ if $z_l=1$.\\
For $2\leq j\leq l-1$, we have $u_j\ge 2, v_j\geq 2$, $u_j+v_j\geq
6$ since $l\geq 5$, so\\ $(u_{j}+v_j)z_j\leq (u_jv_j)z_j\leq
(u_jv_j)z_j^2$ and $(u_{j}+v_j)z_j\leq (u_jv_j)z_j^2-2$ if
$z_j\geq 1$.
\end{proof}

\section{Algebraic surfaces with quotient singularities}
\subsection{} A singularity $p$ of a normal surface $S$ is called a quotient
singularity if the germ  is  locally analytically isomorphic to
$(\mathbb{C}^2/G,O)$  for some nontrivial finite subgroup $G$ of
$GL_2(\mathbb{C})$ without quasi-reflections. Brieskorn classified
all such finite subgroups of $GL(2, \mathbb{C})$ [Bri].

 Let $S$ be a normal projective surface with quotient singularities and $$f : S'
\rightarrow S$$ be a minimal resolution of $S$. It is well-known
that quotient singularities are log-terminal
 singularities. Thus one can write $$K_{S'} \underset{num}{\equiv} f^{*}K_S -
 \sum_{p \in Sing(S)}{\mathcal{D}_p},$$ where $\mathcal{D}_p = \sum(a_jA_j)$ is an effective
 $\mathbb{Q}$-divisor with $0 \leq a_j < 1$  supported on $f^{-1}(p)=\cup A_j$   for each singular point $p$.
Intersecting the formula with $\mathcal{D}_p$, we get
$$\mathcal{D}_pK_{S'}=-\mathcal{D}_p^2$$
and hence
\[K^2_S = K^2_{S'} - \sum_{p}{\mathcal{D}_p^2}= K^2_{S'} +\sum_{p}{\mathcal{D}_pK_{S'}}.
\]

For each singular point $p$, the coefficients of the
$\mathbb{Q}$-divisor $\mathcal{D}_p$ can be obtained by
 solving the equations given by the adjunction formula
$$\mathcal{D}_pA_j=-K_{S'}A_j=2+A_j^2$$ for each exceptional curve
$A_j\subset f^{-1}(p)$.

When $p$ is a cyclic singularity or order $q$, the coefficients of
$\mathcal{D}_p$ can be expressed in terms of $v_j$ and $u_j$ (see
Notation 2.3) as follows.

\begin{lemma}\label{Dp}
Let $p$ be a cyclic quotient singular point of $S$. Assume that
$f^{-1}(p)$ has $l$ components $A_1, \ldots, A_l$ with $A_i^2 =
-n_i$ forming a string of smooth rational curves
$\overset{-n_1}{\circ}-\overset{-n_2}\circ-\cdots-\overset{-n_l}\circ$.
Then \begin{enumerate} \item $\mathcal{D}_p
=\overset{l}{\underset{j = 1}{\sum}} \big( 1 - \dfrac{v_{j} +
      u_{j}}{q} \big) A_{j},$
      \item $\mathcal{D}_pK_{S'}=-\mathcal{D}_p^2
=\overset{l}{\underset{j = 1}{\sum}} \big( 1 - \dfrac{v_{j} +
      u_{j}}{q} \big) (n_{j}-2),$
\item $\mathcal{D}^2_p = 2l - \overset{l}{\underset{j = 1}{\sum}} n_j + 2 - \dfrac{q_1 +
q_l + 2}{q}.$\\ In particular, if $l = 1$, then $\mathcal{D}^2_p =
-\dfrac{(n_1-2)^2}{n_1}$.
\end{enumerate}
\end{lemma}

\begin{proof} (1) is well known (cf. \cite{Megyesi} or
 Lemma 2.2 of \cite{HK1}).

 (2) follows from (1) and the adjunction formula.

 (3) is also well known (cf. \cite{LW} or Lemma 3.6 of
 \cite{HK1}).
\end{proof}

Also we  recall the orbifold Euler characteristic
$$ e_{orb}(S) := e(S) - \sum_{p \in Sing(S)} \Big ( 1-\frac{1}{|G_p|} \Big ),$$
where $G_p$ is the local fundamental group of $p$.

The following theorem, called the orbifold Bogomolov-Miyaoka-Yau
inequality, is one of the main ingredients in the proof of our
main theorems.

\begin{theorem}[\cite{Sakai}, \cite{Miyaoka}, \cite{KNS}, \cite{Megyesi}]\label{bmy} Let
$S$ be a normal projective surface with quotient singularities
such that $K_S$ is nef. Then
\[
K_{S}^2 \leq 3e_{orb}(S).
\]
In particular, $$ 0\leq e_{orb}(S).$$
\end{theorem}

The weaker inequality holds when $-K_S$ is nef.

\begin{theorem}[\cite{KM}]\label{bmy2} Let
$S$ be a normal projective surface with quotient singularities
such that $-K_S$ is nef. Then $$ 0\leq e_{orb}(S).$$
\end{theorem}

\subsection{} Let $S$ be a normal projective surface with quotient singularities and $f : S'
\rightarrow S$ be a minimal resolution of $S$. It is well-known
that the torsion-free part of the second cohomology group,
$$H^2(S', \mathbb{Z})_{free} := H^2(S', \mathbb{Z})/(torsion)$$  has a
lattice structure which is unimodular. For a quotient singular
point $p\in S$, let $$R_p\subset H^2(S', \mathbb{Z})_{free}$$ be
the sublattice of $H^2(S', \mathbb{Z})_{free}$ spanned by the
numerical classes of the components of $f^{-1}(p)$. It is a
negative definite lattice, and its discriminant group
$${\rm disc}(R_p):={\rm Hom}(R_p, \mathbb{Z})/R_p$$ is isomorphic to the
abelianization $G_p/[G_p, G_p]$ of the local fundamental group
$G_p$. In particular, the absolute value $|\det(R_p)|$ of the
determinant
 of the intersection matrix of $R_p$ is equal to the order
$|G_p/[G_p, G_p]|$. Let
 $$R= \oplus_{p \in Sing(S)} R_p\subset H^2(S', \mathbb{Z})_{free}$$ be the sublattice of $H^2(S', \mathbb{Z})_{free}$ spanned by the
numerical classes of the exceptional curves of $f:S'\to S$. We
also consider the sublattice $$R+\langle K_{S'}\rangle\subset
H^2(S', \mathbb{Z})_{free}$$ spanned by $R$ and the canonical
class $K_{S'}$. Note that $${\rm rank}(R)\le {\rm rank}(R+\langle
K_{S'}\rangle)\le{\rm rank}(R)+1.$$

\begin{lemma}[\cite{HK1}, Lemma 3.3]\label{detR} Let $S$ be a normal projective surface with quotient singularities and $f : S'
\rightarrow S$ be a minimal resolution of $S$. Then the following
hold true.
\begin{enumerate}
\item ${\rm rank}(R+\langle K_{S'}\rangle)={\rm rank}(R)$ if and
only if $K_S$ is numerically trivial. \item $\det(R+\langle
K_{S'}\rangle)=\det(R)\cdot K_S^2$ if $K_S$ is not numerically
trivial. \item If in addition $b_2(S)=1$ and $K_S$ is not
numerically trivial, then $R+\langle K_{S'}\rangle$ is a
sublattice of finite index in the unimodular lattice $H^2(S',
\mathbb{Z})_{free}$, in particular $|\det(R+\langle
K_{S'}\rangle)|$ is a nonzero square number.
\end{enumerate}
\end{lemma}

We denote the number $|\det(R+\langle K_{S'}\rangle)|$ by $D$,
i.e., we define
$$D:=|\det(R+\langle K_{S'}\rangle)|.$$
The following is well known.

\begin{lemma}[]\label{ep} Assume that $p$ is a cyclic singularity such that
$f^{-1}(p)$ has $l$ components $A_{1} \ldots, A_l$ with ${A_i}^2 =
-n_i$ forming a string of smooth rational curves
$\overset{-n_1}{\circ}-\overset{-n_2}\circ-\cdots-\overset{-n_l}\circ$.
Then ${\rm disc}(R_p)$ is a cyclic group generated by
$$e_p:=A_l^* = -\frac{1}{q}\underset{i = 1}{
\overset{l}{\sum}} u_{i} A_{i}$$ where $u_i=|\det[n_1, n_2,
\ldots, n_{i-1}]|$ as in Notation \ref{uv-n}. It has the property
that
$$e_pA_l=1,\,\,\, e_pA_j=0
\,\,\,(1\le j\le l-1)\,\,\,{\rm and}\,\,\, e_p^2 = -
\frac{u_l}{q}= - \frac{q_l}{q}.$$
\end{lemma}

The following will be also useful in our proof.

\begin{lemma}[\cite{HK2}, Lemma 2.5]\label{coprime} Let $S$ be a $\mathbb{Q}$-homology projective plane with cyclic singularities
 such that $H_1(S^0,\mathbb{Z}) = 0$. Let $f:S' \rightarrow S$ be a minimal resolution. Then
\begin{enumerate}
\item $H^2(S', \mathbb{Z})$ is torsion free, i.e., $H^2(S',
\mathbb{Z})=H^2(S', \mathbb{Z})_{free}$, \item $R$ is a primitive
sublattice of the unimodular lattice $H^2(S', \mathbb{Z})$, \item
 ${\rm disc}(R)$ is a cyclic group, in particular, the orders $|G_p| = |\det(R_p)|$ are pairwise relatively prime, \item $K_S$ is not numerically
 trivial, i.e., $K_S$ is either ample or anti-ample,
\item $D=|\det(R)|K_S^2$ and is a nonzero square number,
 \item the Picard group $Pic(S')$ is generated over $\mathbb{Z}$ by the
exceptional curves and a $\mathbb{Q}$-divisor $M$ of the form
$$M = \frac{1}{\sqrt{D}} f^*K_S +  \underset{p \in Sing(S)}{\sum}
b_p e_p $$  for some integers $b_p$, where $e_p$ is the generator
of ${\rm disc}(R_p)$ as in Lemma \ref{ep}.
\end{enumerate}
\end{lemma}

Finally we generalize Lemma \ref{coprime} to the case without the
condition $H_1(S^0,\mathbb{Z})=0$. We will encounter this general
situation later in our proof (see Sections 5 and 6).

Let $S$ be a $\mathbb{Q}$-homology projective plane with cyclic
singularities and $f:S' \rightarrow S$ be a minimal resolution.
Denote by $Pic(S')_{free}$ the group of numerical equivalence
classes of divisors, i.e.,
$$Pic(S')_{free}:=Pic(S')/(torsion).$$ With the intersection
pairing, $Pic(S')_{free}$ becomes a unimodular lattice isometric
to $H^2(S', \mathbb{Z})_{free}$. Denote  by $$\bar{R}\subset
Pic(S')_{free}$$ the primitive closure of $R\subset
Pic(S')_{free}$, the sublattice spanned by the numerical
equivalence classes of exceptional curves of $f$.

\begin{lemma}\label{general} Let $S$ be a $\mathbb{Q}$-homology projective plane with cyclic singularities
and $f:S' \rightarrow S$ be a minimal resolution. Assume that
$K_S$ is not numerically
 trivial. Then the following hold true.
\begin{enumerate}
\item $D=|\det(R)|
K_S^2$ and is a nonzero square number.
\item ${\rm disc}(\bar{R})$ is a cyclic group of order $|\det(\bar{R})|=\frac{|\det(R)|}{c^2}$ where $c$ is the order of $\bar{R}/R$.
\item Define $$D':=|\det(\bar{R})|
K_S^2=\frac{D}{c^2}.$$ Then $Pic(S')_{free}$ is generated over
$\mathbb{Z}$ by the numerical equivalence classes of exceptional
curves, an element $T\in Pic(S')_{free}$ giving a generator of
$\bar{R}/R$ and a $\mathbb{Q}$-divisor of the form
$$M = \frac{1}{\sqrt{D'}} f^*K_S +  z, $$ where $z$ is a generator of ${\rm
disc}(\bar{R})$, hence of the form $z= \underset{p \in
Sing(S)}{\sum} b_p e_p$ for some integers $b_p$, where $e_p$ is
the generator of ${\rm disc}(R_p)$ as in Lemma \ref{ep}.
\item For each singular point $p$, denote by
$A_{1, p}, A_{2, p}, \ldots, A_{l_p, p}$ the exceptional curves of
$f$ at $p$ and by $q_p$ the order of the local fundamental group
at $p$. Then every element $E\in Pic(S')_{free}$ can be written
uniquely as
\begin{equation}\label{E}
E = mM + \underset{p \in Sing(S)}{\sum} \overset{l_p}{\underset{i
= 1}{\sum}} a_{i, p} A_{i,p}
\end{equation}
 for some integer $m$
and some $a_{i,p}\in \frac{1}{c}\mathbb{Z}$ for all $i, p$.
\item $E$ is supported on $f^{-1}(Sing(S))$ if and only if $m=0$.
Moreover, if $E$ is effective (modulo a torsion) and not supported
on $f^{-1}(Sing(S))$, then $m>0$ when $K_S$ is ample, and $m<0$
when $-K_S$ is ample.
\end{enumerate}
\end{lemma}

\begin{proof} (1) follows from Lemma \ref{detR}.

(2) is well known.

(3) We slightly modify the proof of \cite{HK2}, Lemma 2.5. Here,
$R^\perp$ is generated by $$v:=\frac{\sqrt{D'}}{K_S^2}
f^*K_S=\frac{|\det(\bar{R})|}{\sqrt{D'}} f^*K_S,$$ ${\rm
disc}(R^\perp)$ is generated by $$\frac{1}{\sqrt{D'}} f^*K_S,$$
and
$$Pic(S')_{free}/(R^\perp\oplus \bar{R})\subset {\rm
disc}(R^\perp\oplus \bar{R})$$ is an isotropic subgroup of order
$|\det(\bar{R})|$ of ${\rm disc}(R^\perp\oplus \bar{R})$, hence is
generated by an element $$M\in {\rm disc}(R^\perp\oplus \bar{R})$$
of order $|\det(\bar{R})|$. Moreover $M$ is the sum of a generator
of ${\rm disc}(R^\perp)$ and a generator of ${\rm disc}(\bar{R})$,
since $Pic(S')_{free}$ is unimodular. By replacing $M$ by $kM$ for
a suitable choice of an integer $k$, we get $M$ of the desired
form. We have shown that $Pic(S')_{free}$ is generated over
$\mathbb{Z}$ by $v$, $\bar{R}$ and $M$. Note that
$$|\det(\bar{R})|M\equiv v \,\,\, {\rm  modulo} \,\,\,\bar{R},$$ i.e., $v$ is generated by $M$ and $\bar{R}$.
Finally $\bar{R}$ is generated over $\mathbb{Z}$ by $R$ and $T$.

(4) By (3) $E$ is a $\mathbb{Z}$-linear combination of $M$, $T$,
and $A_{i,p}$. Since $cT\in R$, the result follows.

(5) The first assertion is obvious. For the second, note that
$$E(f^*K_S)=mM(f^*K_S)=\dfrac{m}{\sqrt{D'}}K_S^2.$$
\end{proof}

\section{Curves on the minimal resolution}

Throughout this section, we denote by $S$ a $\mathbb{Q}$-homology
projective plane with cyclic singularities and by $f:S'
\rightarrow S$ its minimal resolution, and assume that  $K_S$ is
not numerically trivial. But we do not assume that
$H_1(S^0,\mathbb{Z}) = 0$. So, the orders of singularities may not
be pairwise relatively prime.

Let $E$ be a divisor on $S'$. Then by Lemma \ref{general}(4), the
numerical equivalence class of $E$ can be written as the form
\eqref{E}. The coefficients of $E$ in \eqref{E} and the
intersection numbers $EA_{j,p}$ are related as follows. Here $u_j$
and $v_j$ are as in Notation 2.3.

\begin{lemma}\label{tilde} Fix $p \in Sing(S)$. Then for $j = 1, \ldots, l_p$
$$\dfrac{u_{j,p}}{q_p} mb_p -a_{j,p}=
\overset{j}{\underset{k = 1}{\sum}} \dfrac{v_{j,p}
  u_{k,p}}{q_p}(EA_{k,p}) + \overset{l_p}{\underset{k = j+1}{\sum}}
\dfrac{v_{k,p}  u_{j,p}}{q_p}(EA_{k,p}).$$
\end{lemma}

\begin{proof}
Note that, by Lemma \ref{ep}, for each $p\in Sing(S)$
$$MA_{j,p}=0\,\,\, {\rm for}\,\,\, j=1, \ldots, l_p-1,\,\,\, {\rm
and} \,\,\, MA_{l_p,p}=b_p.$$ We fix $p$ and, for simplicity, omit
the subscript $p$. Thus we obtain the following system of
equalities:
$$\begin{array}{l}
EA_{1}=-n_{1}a_{1} +a_{2}\\
EA_{2}=a_{1}-n_{2} a_{2}+a_{3}\\
EA_{3}=a_{2}-n_{3} a_{3}+a_{4}\\
\ldots\\
EA_{l-1}=a_{l-2}-n_{l-1} a_{l-1}+a_{l}\\
EA_{l}=a_{l-1}-n_{l} a_{l} + mb.\\
\end{array}$$
It implies that
$$\begin{array}{l}
a_{1} = \frac{1}{n_{1}}a_{2} - \frac{1}{n_{1}}
EA_{1}=\frac{u_{1}}{u_{2}}a_{2} - \frac{1}{u_{2}}
EA_{1}\\
a_{2} =\frac{u_{2}}{u_{3}}a_{3} - \frac{1}{u_{3}}
EA_{1}- \frac{u_{2}}{u_{3}}EA_{2}\\
\ldots\\
a_{j} =\frac{u_{j}}{u_{j+1}}a_{j+1} - \frac{1}{u_{j+1}}
EA_{1}-\ldots -\frac{u_{j}}{u_{j+1}}EA_{j}\\
\ldots\\
a_{l-1} =\frac{u_{l-1}}{u_{l}}a_{l} - \frac{1}{u_{l}}
EA_{1}-\ldots -\frac{u_{l-1}}{u_{l}}EA_{l-1}\\
a_{l} =\frac{u_{l}}{q}mb - \frac{1}{q} EA_{1}-\ldots
-\frac{u_{l}}{q}EA_{l}=\frac{u_{l}}{q}mb -\overset{l}{\underset{k
= 1}{\sum}} \frac{v_{l}u_{k}}{q}EA_{k}.\\
\end{array}$$
Plugging the last equation into the above equation for $a_{l-1}$,
we obtain
\begin{eqnarray*}
     a_{l-1}
    & = & \frac{u_{l-1}}{u_{l}}(\frac{u_{l}}{q}mb - \frac{1}{q} EA_{1}-\ldots
-\frac{u_{l}}{q}EA_{l}) - \frac{1}{u_{l}}
EA_{1}-\ldots -\frac{u_{l-1}}{u_{l}}EA_{l-1}\\
& = & \frac{u_{l-1}}{q}mb - \overset{l-1}{\underset{k = 1}{\sum}}
\frac{(u_{l-1}+q)u_k}{qu_l}EA_{k}-\frac{u_{l-1}}{q}EA_{l}.
    \end{eqnarray*}
By Lemma \ref{cf}(2), $$u_{l-1}+q=v_lu_{l-1}+q=v_{l-1}u_l,$$ so
the required equation for $a_{l-1}$ follows.

Next, plugging the required equation for $a_{l-1}$ into the above
equation for $a_{l-2}$, we obtain the required equation for
$a_{l-2}$. Others can be obtained similarly.
\end{proof}

Now we express the intersection numbers $EK_{S'}$ and $E^2$ in
terms of the intersection numbers $EA_{j,p}$ of $E$ and the
exceptional curves $A_{j,p}$.

\begin{proposition}\label{int}
 Let $E$ be a divisor on $S'$. Write $($the numerical equivalence class of$)$ $E$ as the form \eqref{E}. Then the following hold
true.
    \begin{enumerate}
     \item $
          EK_{S'}
 = \dfrac{m}{\sqrt{D'}}K^2_S - \underset{p}{\sum} \overset{ l_p}{\underset{j = 1}{\sum}}
      \big( 1 - \dfrac{v_{j,p} + u_{j,p}}{q_p} \big) EA_{j,p}.$

\medskip\noindent If $EA_{j,p}\geq 0$ for all $p$ and $j$, then\\
$ EK_{S'} \leq \dfrac{m}{\sqrt{D'}}K^2_S
-\underset{p}{\sum}\overset{ l_p}{ \underset{j=1}{\sum}}
\big(1-\dfrac{2}{n_{j,p}} \big)EA_{j,p}.$
        \item $ E^2 =\dfrac{m^2}{D'}K^2_S-  \underset{p}{\sum}
\overset{l_p}{\underset{j = 1}{\sum}}\Big(\overset{j}{\underset{k
= 1}{\sum}} \dfrac{v_{j,p}
  u_{k,p}}{q_p}(EA_{k,p}) + \overset{l_p}{\underset{k = j+1}{\sum}}
\dfrac{v_{k,p}  u_{j,p}}{q_p}(EA_{k,p})
 \Big)EA_{j,p}.$

\medskip\noindent If $EA_{j,p}\geq 0$ for all $p$ and $j$, then\\ $E^2 \leq
\dfrac{m^2}{D'} K^2_S - \underset{p}{\sum} \overset{
l_p}{\underset{j = 1}{\sum}} \dfrac{v_{j,p} u_{j,p}}{q_p}
(EA_{j,p})^2.$
         \item If, for each $p \in
                Sing(S)$, $E$ has a non-zero intersection number with at most $2$ components of $f^{-1}(p)$,
                i.e., $EA_{j,p}=0$ for $j\neq s_p, t_p$ for some $s_p$ and $t_p$ with $1\le s_p< t_p\le l_p$,
                then\\
                $ E^2  =   \dfrac{m^2}{D'} K^2_S - \underset{p}{\sum} \Big(  \dfrac{v_{s_p} u_{s_p}}{q_p}
                (EA_{s_p})^2+ \dfrac{v_{t_p} u_{t_p}}{q_p} (EA_{t_p})^2
                + \dfrac{2 v_{t_p} u_{s_p}}{q_p} (EA_{s_p})(EA_{t_p}) \Big) .$
    \end{enumerate}
\end{proposition}


\begin{proof}
(1) Note that
$$K_{S'}=f^*(K_S)-\underset{p \in Sing(S)}{\sum}
\overset{l_p}{\underset{j = 1}{\sum}} \big( 1 - \dfrac{v_{j,p} +
      u_{j,p}}{q_p} \big) A_{j,p}.$$
      Intersecting both sides with $E$, we get
$$EK_{S'} = Ef^*(K_S)-\underset{p}{\sum}
\overset{l_p}{\underset{j = 1}{\sum}} \big( 1 - \dfrac{v_{j,p} +
      u_{j,p}}{q_p} \big)E A_{j,p}.$$
Intersecting both sides of $$E = mM + \underset{p}{\sum}
\overset{l_p}{\underset{i = 1}{\sum}} a_{i, p} A_{i,p}$$ with
$f^*(K_S)$, we get
$$Ef^*(K_S)=mMf^*(K_S)=\frac{m}{\sqrt{D'}}f^*(K_S)^2=\frac{m}{\sqrt{D'}}K_S^2.$$
This proves the equality.

Note that
$$\begin{array}{llll} n_j(v_j+u_j)&=&(v_{j+1}+v_{j-1})+(u_{j+1}+u_{j-1}) &({\rm Lemma}\,\, \ref{cf}(1))\\
&=&(u_{j+1}+v_{j+1})+(u_{j-1}+v_{j-1}) \le 2q &({\rm Lemma}\,\,
\ref{cf}(5)).\end{array}$$ Thus
$$\dfrac{v_{j,p} + u_{j,p}}{q_p} \le \dfrac{2}{n_{j,p}}$$ for all
$p$ and $j$. This gives the inequality.

 (2) Intersecting both sides of $$E = mM + \underset{p}{\sum}
\overset{l_p}{\underset{j = 1}{\sum}} a_{j, p} A_{j,p}$$ with $E$,
we get $$E^2 = mEM + \underset{p}{\sum} \overset{l_p}{\underset{j
= 1}{\sum}} a_{j, p} EA_{j,p}.$$  Intersecting both sides of
$$M = \frac{1}{\sqrt{D'}} f^*K_S +  \underset{p}{\sum} b_p e_p$$ with $E$, we get
$$\begin{array}{llll}  mEM &=& \dfrac{m}{\sqrt{D'}}Ef^*(K_S)+ m\underset{p}{\sum} b_p Ee_p&\\ &=& \dfrac{m}{\sqrt{D'}}\dfrac{m}{\sqrt{D'}}K_S^2+
 m\underset{p}{\sum} b_p(mMe_p+a_{l,p})&\\
 &=&\dfrac{m^2}{D'}K_S^2+  m\underset{p}{\sum} b_p(mb_pe_p^2+a_{l,p})&\\
& =& \dfrac{m^2}{D'}K_S^2+  m\underset{p}{\sum}
b_p(-\dfrac{mb_pu_{l,p}}{q}+a_{l,p})&({\rm Lemma}\,\, \ref{ep})\\
& =& \dfrac{m^2}{D'}K_S^2-  m\underset{p}{\sum}
b_p(\overset{l_p}{\underset{k =
1}{\sum}}\dfrac{v_{l,p}u_{k,p}}{q}EA_{k,p})&({\rm Lemma}\,\,
\ref{tilde}).\end{array}$$ Thus
$$\begin{array}{llll}  E^2
& =& \dfrac{m^2}{D'}K_S^2-  m\underset{p}{\sum}
b_p(\overset{l_p}{\underset{j =
1}{\sum}}\dfrac{v_{l,p}u_{j,p}}{q}EA_{j,p})+ \underset{p}{\sum}
\overset{l_p}{\underset{j = 1}{\sum}} a_{j, p} EA_{j,p}\\
& =& \dfrac{m^2}{D'}K_S^2-  \underset{p}{\sum}
\overset{l_p}{\underset{j = 1}{\sum}}(\dfrac{mb_pu_{j,p}}{q}-
a_{j, p}) EA_{j,p}.\end{array}$$ Now the equality follows from
Lemma \ref{tilde}.

If $EA_{j,p}\geq 0$ for all $p$ and $j$, then
$$\overset{j}{\underset{k =
1}{\sum}} \dfrac{v_{j,p}
  u_{k,p}}{q_p}(EA_{k,p}) + \overset{l_p}{\underset{k = j+1}{\sum}}
\dfrac{v_{k,p}  u_{j,p}}{q_p}(EA_{k,p})
 \ge \dfrac{v_{j,p}
  u_{j,p}}{q_p}(EA_{j,p}), $$
  so the inequality follows.

(3) If  $EA_{j,p}=0$ for $j\neq s_p, t_p$ for some $s_p$ and $t_p$
with $1\le s_p< t_p\le l_p$,
                then
                $$\begin{array}{ll}&\overset{l_p}{\underset{j = 1}{\sum}}\Big(\overset{j}{\underset{k =
1}{\sum}} \dfrac{v_{j,p}
  u_{k,p}}{q_p}EA_{k,p} + \overset{l_p}{\underset{k = j+1}{\sum}}
\dfrac{v_{k,p}  u_{j,p}}{q_p}EA_{k,p}
 \Big)EA_{j,p}\\&=\Big(\dfrac{v_{s_{p}}
  u_{s_{p}}}{q_p}EA_{s_{p}} + \dfrac{v_{t_{p}}
  u_{s_{p}}}{q_p}EA_{t_{p}}\Big)(EA_{s_{p}})+\Big(\dfrac{v_{t_{p}}
  u_{s_{p}}}{q_p}EA_{s_{p}} + \dfrac{v_{t_{p}}
  u_{t_{p}}}{q_p}EA_{t_{p}}\Big)(EA_{t_{p}}),
\end{array}$$
so the equality follows from (2).
\end{proof}

Let $$L=L_S:={\rm rank}(R)$$ be the number of the irreducible
exceptional curves of $f:S' \rightarrow S$. We have
$$b_2(S')=1+L.$$
Note that $S'$ has $H^1(S', \mathcal{O}_{S'})=H^2(S',
\mathcal{O}_{S'})=0$. Thus by Noether formula,
\[K^2_{S'}= 12- e(S')=10-b_2(S')=9-L.\]

\begin{lemma}\label{bound}
Let $S$ be a $\mathbb{Q}$-homology projective plane with cyclic
singularities. Assume that $K_S$ is not numerically trivial.
Assume that $S$ is not rational. If $L
> 9$, then there is a $(-1)$-curve $E$ on $S'$ of the form
\eqref{E} with $0 < m \leq \dfrac{\sqrt{D'}}{L-9}.$
\end{lemma}

\begin{proof} Since $S$ is not rational and $K_S$ is not numerically trivial, $K_S$ is ample. Thus $m>0$ for any $(-1)$-curve $E$ by Lemma
\ref{general}(5).\\ Since $K^2_{S'}=9-L < 0$, $S'$ is not a
minimal surface. Let $$g:S'=S_k \rightarrow S_{k-1}\rightarrow
S_{k-2}\rightarrow\cdots\rightarrow S_{1}\rightarrow
S_{0}=S_{min}$$ be a morphism of $S'$ to its minimal model. Since
$K_{S_{min}}^2\geq 0$, we see that $$k\geq L-9.$$ Also one can
write
$$ K_{S'} = g^*K_{S_{min}} + \overset{k}{\underset{i=1}{\sum}}E_i$$ where
$E_i$ is the total transform of the exceptional curve of the
blowup $S_i\rightarrow S_{i-1}$. Note that $E_1, \ldots, E_k$ are
effective divisors, not necessarily irreducible, satisfying $E_i^2
= -1$ and $E_iE_j=0$ for $i\neq j$. \\ Let $m_0$ be the leading
coefficient of $g^*K_{S_{min}}$ written in the form \eqref{E}.
Since $S$ is not rational, $K_{S_{min}}$ is a nef
$\mathbb{Q}$-divisor on $S_{min}$, so $g^*K_{S_{min}}$ is a nef
$\mathbb{Q}$-divisor on $S'$. Since $K_S$ is ample, this implies
that $ m_0\ge 0$. Let $m_i$ be the leading coefficient of $E_i$
written in the form \eqref{E}. Note that $\sqrt{D'}$ is the
leading coefficient of $K_{S'}$ written in the form \eqref{E}.
Thus
$$\sqrt{D'}= m_0+ \overset{k}{\underset{i=1}{\sum}}m_i.$$ If $E_s$ is a $(-1)$-curve and is a component of $E_t$ for
some $t\neq s$, then one can write $E_t = aE_s + F$ where $a \geq
1$ is an integer and $F$ is an effective divisor. It follows that
$m_t \ge am_s\geq m_s.$ Let $$m:=\min\{m_1, m_2,\ldots, m_k\}.$$
Then there is an irreducible member $E$ among $E_1,\ldots, E_k$
whose leading coefficient is $m$. It is a $(-1)$-curve, and
$$\sqrt{D'}= m_0+ \overset{k}{\underset{i=1}{\sum}}m_i\ge \overset{k}{\underset{i=1}{\sum}}m_i\ge km\ge (L-9)m.$$
\end{proof}

\section{First reduction steps for the cases with $|Sing(S)|\ge 4$}
Let $S$ be a $\mathbb{Q}$-homology projective plane with cyclic
quotient singularities such that $H_1(S^0, \mathbb{Z}) = 0$. By
 Lemma \ref{coprime}(3), the orders of
singularities are pairwise relatively prime. Since $e_{orb}(S)\ge
0$ (Theorems \ref{bmy} and \ref{bmy2}), one can immediately see
that $S$ can have at most $4$ singular points (also see
\cite{HK1}, \cite{Kol08}).

Assume that $|Sing(S)|=4$. Then we enumerate all possible
$4$-tuples of orders of local fundamental groups:

\bigskip
\begin{enumerate}
\item $(2,3,5,q)$, $q \geq 7$, $\gcd(q, 30) = 1$, \item
$(2,3,7,q)$, $11 \leq q \leq 41$, $\gcd(q, 42) = 1$, \item
$(2,3,11,13)$.
\end{enumerate}
\bigskip

For (2) and (3), there are exactly 1092 different possible types
for $R$, the sublattice of $H^2(S', \mathbb{Z})_{free}$ generated
by all exceptional curves of the minimal resolution $f:S'\to S$.
There are 2 types, $[3]$, $[2,2]$, of order 3;  4 types, $[7]$,
$[4,2]$, $[3,2,2]$, $A_6$, of order 7; $\frac{\phi(q)}{2}+1$ types
of order $q$, so the total number of types of $R$ for the case
$(2,3,7,q)$ is
$$2\times 4\times \big(\dfrac{\phi(q)}{2}+1\big)=4(\phi(q)+2),$$ where
$\phi$ is the Euler function. Here we identify
$\frac{1}{q}(1,q_1)$ with $\frac{1}{q}(1,q_l)$. By Lemma
\ref{coprime}(5), the number
$$D=|\det(R)|K_{S}^2$$ must be a nonzero square number. Among the
1092 cases, a computer calculation of the number $D$ shows that
only 24 cases satisfy this property. Table \ref{finite0} describes
these 24 cases.

The number $D$ can be computed as follows. First note that
$$|\det(R)|={\rm the\,\, product\,\, of\,\, orders}.$$ To compute $K^2_S$, we
use the equality from 3.1,
\[K^2_S = K^2_{S'} + \sum_{p}{\mathcal{D}_pK_{S'}}.
\]
By Noether formula,
\[K^2_{S'}=9-L\]
where $L:={\rm rank}(R)$ is the number of the exceptional curves of $f$.\\
Finally the intersection number $\mathcal{D}_pK_{S'}$ is given in
Lemma \ref{Dp}.

\begin{remark}
None of the 24 cases of Table \ref{finite0} can be ruled out by
any further lattice theoretic argument. In fact, in each case the
lattice $R$ can be embedded into a unimodular lattice
$I_{1,L}$(odd) or $II_{1,L}$(even) of signature $(1,L)$. This can
be checked by the local-global principle and the computation of
$\epsilon$-invariants (see e.g., \cite{HK1} Section 6).
\end{remark}

\begin{table}[ht]
\caption{}\label{finite0}
\renewcommand\arraystretch{1.5}
\noindent\[
\begin{array}{|c|l|l|c|l|c|}
\hline
\textrm{No.} & \textrm{Type\,\,of\,\,$R$} &\textrm{orders} & K^2_S & &3e_{orb}(S)\\
 \hline

1 &  [2]+  A_2 + [7] + [13]&(2,3,7,13) & \frac{1536}{91}   & >  & \frac{29}{182}\\
\hline

2 &[2]+  A_2 + [7] + [3,2,2,2,2,2,2,2,2]&(2,3,7,19) & \frac{6}{133} & < & \frac{23}{266}\\
\hline

3 &[2]+  A_2 + [7] + [5, 4]&(2,3,7,19) & \frac{1350}{133}  & >  & \frac{23}{266}\\
\hline

4 &[2]+  A_2 + [7] + [3,4,2]&(2,3,7,19) & \frac{1014}{133}  & >  & \frac{23}{266}\\
\hline

5 & [2]+ A_2 + [4, 2] + [2,2,4,2,2,2]&(2,3,7,31) & \frac{150}{217}  & >  & \frac{11}{434}\\
\hline

6 & [2]+ A_2 + [4, 2] + [6,2,2,2,2,2]&(2,3,7,31) & \frac{486}{217} & > &  \frac{11}{434}\\
\hline

7 & [2]+ [3] + [3, 2, 2] + [4,2,2,2,3]&(2,3,7,29) & \frac{968}{609}  & > &  \frac{13}{406}\\
\hline

8 & [2]+ A_2 + [3, 2, 2] + [7,2,2,2]&(2,3,7,25) & \frac{24}{7}  & > &  \frac{17}{350}\\
\hline

9 & [2]+ A_2 + [7] + [2,2,3,2,2,2,2,2,2]&(2,3,7,31) & \frac{54}{217}   & >  & \frac{11}{434}\\
\hline

10 & [2]+ [3] + [4, 2] + [3,3,2,2,3]&(2,3,7,41) & \frac{2888}{861}& > &  \frac{1}{574}\\
\hline

11 & [2]+ A_2 + [3, 2, 2] + [7,2,2,2,2,2]&(2,3,7,37) & \frac{384}{259} & >  & \frac{5}{518}\\
\hline

12 &[2]+  A_2 + [4, 2] + [11,2,2]&(2,3,7,31) & \frac{2166}{217} & >  & \frac{11}{434}\\
\hline

13 &[2]+  [3] + A_6 + [2,6,2,2]&(2,3,7,29) & \frac{56}{87}  & >  & \frac{13}{406}\\
\hline

14 &[2]+  [3] + [3, 2, 2] + [4,3]&(2,3,7,11) & \frac{1058}{231} & >  & \frac{31}{154}\\
\hline

15 &[2]+  [3] + [3, 2, 2] + [3,2,2,2,2]&(2,3,7,11) & \frac{50}{231}  & > &  \frac{31}{154}\\
\hline

16 &[2]+  [3] + [3, 2, 2] + [4,2,2,3]&(2,3,7,23) & \frac{1250}{483}  & >  & \frac{19}{322}\\
\hline

17 &[2]+  [3] + [3, 2, 2] + [6,5]&(2,3,7,29) & \frac{5000}{609} & >  & \frac{13}{406}\\
\hline

18 & [2]+ A_2 + [3, 2, 2] + [3,5,2]&(2,3,7,25) & \frac{24}{7}   & > &  \frac{17}{350}\\
\hline

19 & [2]+ A_2 + [3, 2, 2] + [13,2]&(2,3,7,25) & \frac{1944}{175}    & >  & \frac{17}{350}\\
\hline

20 & [2]+ A_2 + [4, 2] + [4,2,2,2]&(2,3,7,13) & \frac{216}{91} & >  & \frac{29}{182}\\
\hline

21 & [2]+ A_2 + [4, 2] + [5,2,2]&(2,3,7,13) & \frac{384}{91}  & > &  \frac{29}{182}\\
\hline

22 & [2]+ A_2 + [4, 2] + [4,2,2,2,2,2]&(2,3,7,19) & \frac{54}{133}   & > &  \frac{23}{266}\\
\hline

23 & [2]+ [3] + [3,2,2,2,2] + [4,2,2,2]&(2,3,11,13) & \frac{8}{429}   & >  & \frac{1}{286}\\
\hline

24 &  [2]+ [3] + [3,2,2,2,2]  + [5,2,2]&(2,3,11,13) & \frac{800}{429} & > &  \frac{1}{286}\\
\hline

\end{array}
\]
\end{table}

\begin{lemma}\label{24} In all cases of Table \ref{finite0}
except the second case, $-K_S$ is ample.\\
In the second case, $S$ is rational.
\end{lemma}

\begin{proof} The 23 cases do not satisfy the inequality $K_S^2\le 3e_{orb}(S)$ in Theorem \ref{bmy}. Thus the first assertion
follows.

Consider the second case $A_1 +A_2 + [7] + [3,2,2,2,2,2,2,2,2]$.
In this case, $$K^2_S = \frac{6}{133}, \quad D=|\det(R)|K_S^2=36,
\quad L=13.$$ Suppose that $S$ is not rational. By Lemma
\ref{bound}, $S'$ contains a $(-1)$-curve $E$ with $0< m\leq
\frac{\sqrt{D}}{L-9}=\frac{6}{4}$, i.e., $m=1$. By Proposition
\ref{int}(1), we obtain
$$\underset{p}{\sum}\underset{j}{\sum} \Big(1-\frac{v_{j,p} + u_{j,p}}{q_p}  \Big)(EA_{j,p}) = -EK_{S'} +
\frac{m}{\sqrt{D}} K^2_S=1+
\frac{1}{6}\cdot\frac{6}{133}=\frac{134}{133}.$$ Looking at Table
\ref{finite1},
\begin{table}[ht]
\caption{}\label{finite1}
\renewcommand\arraystretch{1.5}
\noindent\[
\begin{array}{|c|c|c|c|c|c|c|c|c|c|c|c|c|c|}
\hline
 & [2] & \multicolumn{2}{|c|}{[2,2]}  & [7] & \multicolumn{9}{|c|}{[3,2,2,2,2,2,2,2,2]}  \\   \hline
j&1&1&2&1& 1 &2&3&4&5&6&7&8&9\\
 \hline
1-\frac{v_j + u_{j}}{q} & 0 &0&0& \frac{5}{7} & \frac{9}{19}&
\frac{8}{19}&  \frac{7}{19}&  \frac{6}{19}&  \frac{5}{19}&
\frac{4}{19}&  \frac{3}{19}&
  \frac{2}{19}&  \frac{1}{19}\\
  \hline
\end{array}
\]
\end{table}
we see that there are non-negative integers $x, y$ such that
$$\frac{5x}{7}+\frac{y}{19}=\frac{134}{133}.$$
But it is easy to check that this equation has no solution.
\end{proof}

\medskip

 Next we consider the cases: $(2,3,5,q)$, $q \geq 7$, $\gcd(q, 30) = 1$.

 \begin{lemma}\label{noA2} In the cases $(2,3,5,q)$, $q \geq 7$, $\gcd(q, 30) = 1$, the order $3$ singularity must be of type
$\frac{1}{3}(1,1)$.
\end{lemma}

\begin{proof}
Suppose that it is of type $A_2$. We divide the proof into 3 cases
according to the type of the third singularity.

\medskip
Case 1: $A_1+A_2+A_4+\frac{1}{q}(1,q_1)$.  In this case $$K^2_S=
\overset{l}{\underset{j = 1}{\sum}} n_j - 3l+
\frac{q_1+q_l+2}{q},$$ and $$
D=30\{q_1+q_l+(\overset{l}{\underset{j = 1}{\sum}} n_j - 3l)q +
2\}.$$ Since $D$
        is a square number, $3$ divides $q_1 + q_l + (tr -3l)q+2\equiv q_1 + q_l + (tr)q+2$. Then, by Proposition \ref{det-trace},
        $q$ is a multiple of $3$, a contradiction.

\medskip
Case 2: $A_1+A_2+\frac{1}{5}(1,2)+\frac{1}{q}(1,q_1)$. In this
case $$K^2_S=\overset{l}{\underset{j = 1}{\sum}} n_j - 3l +
\frac{12}{5} + \frac{q_1+q_l+2}{q},$$ and $$D=6[5(q_1+q_l) + \{
5(\overset{l}{\underset{j = 1}{\sum}} n_j - 3l) + 12 \}q + 10].$$
Thus $3$ divides $5(q_1 + q_l) + \{5(tr -3l)+12\}q+10\equiv -(q_1
+ q_l) -(tr)q+1$. Then, by Proposition \ref{det-trace},
        $q$ is a multiple of $3$, a contradiction.

\medskip
Case 3: $A_1+A_2+\frac{1}{5}(1,1)+\frac{1}{q}(1,q_1)$. In this
case $$K^2_S=\overset{l}{\underset{j = 1}{\sum}} n_j - 3l +
\frac{24}{5} + \frac{q_1+q_l+2}{q},$$ and $$D=6[5(q_1+q_l) + \{
5(\overset{l}{\underset{j = 1}{\sum}} n_j - 3l) + 24 \}q + 10].$$
Thus $3$ divides $5(q_1 + q_l) + \{5(tr -3l)+24\}q+10$. Then, by
Proposition \ref{det-trace},
        $q$ is a multiple of $3$, a contradiction.
\end{proof}

In the following two lemmas, we do not assume that $H_1(S^0,
\mathbb{Z}) = 0$. So the orders may not be pairwise relatively
prime.

\begin{lemma}\label{blow}
Let $S$ be a $\mathbb{Q}$-homology projective plane with exactly
$4$ cyclic singular points $p_1, p_2, p_3, p_4$ of orders
$(2,3,5,q)$, $q \geq 7$. $($We do not assume that $\gcd(q, 30) =
1$.$)$ Regard $\mathcal{F}:=f^{-1}(Sing(S))$ as a reduced integral
divisor on $S'$. Assume that $S'$ contains a $(-1)$-curve $E$.
Then,
$$E.\mathcal{F} \geq 2.$$ The equality holds if and only if
 $E.f^{-1}(p_i)=0$ for $i=1,2,3$
and $E.f^{-1}(p_4)=2$.
\end{lemma}

\begin{proof}
Assume that $E.\mathcal{F} = 1$. Blowing up the intersection
point, then contracting the proper transform of $E$ and the proper
transforms of all irreducible components of $\mathcal{F}$, we
obtain a $\mathbb{Q}$-homology projective plane $\bar{S}$ with $5$
quotient singular points. Then, by \cite{HK1}, the minimal
resolution of $\bar{S}$ is an Enriques surface, hence has no
(-1)-curve, which is a contradiction. This proves that
$E.\mathcal{F} \ge 2$.

\medskip
 Assume that $E.\mathcal{F} = 2$.\\
Suppose that $E$ meets an end component $F$ of $f^{-1}(p_i)$ for
some $1\le i\le 3$.

If $EF=1$, then $EF'=1$ for some other component $F'$ of
$f^{-1}(p_j)$, where $j$ may or may not be $i$. Assume that $E\cap
F\cap F'=\emptyset$. Blowing up the intersection point of $E$ and
$F'$ sufficiently many times, then contracting the proper
transform of $E$ with a string of $(-2)$-curves and the proper
transforms of all irreducible components of $\mathcal{F}$, we
obtain a $\mathbb{Q}$-homology projective plane $\bar{S}$ with $4$
quotient singular points such that $e_{orb}<0$ (see Lemma
\ref{cf}(6)), which violates the orbifold Bogomolov-Miyaoka-Yau
inequality. Assume that $E\cap F\cap F'\neq\emptyset$. Blowing up
the intersection point once, then contracting the proper transform
of $E$ and the proper transforms of all irreducible components of
$\mathcal{F}$, we obtain a $\mathbb{Q}$-homology projective plane
$\bar{S}$ with $6$ quotient singular points, a contradiction to
\cite{HK1}.

If $E$ intersects $F$ at 2 distinct points, then we get a similar
contradiction: blowing up one of the two intersection points of
$E$ and $F$ sufficiently many times, then contracting the proper
transform of $E$ with the adjacent string of $(-2)$-curves and the
proper transforms of all irreducible components of $\mathcal{F}$,
to obtain a $\mathbb{Q}$-homology projective plane $\bar{S}$ with
$4$ quotient singular points such that $e_{orb}<0$.

If $E$ intersects $F$ at 1 point with multiplicity 2, then blowing
up the intersection point twice and then contracting the proper
transform of $E$ with a $(-2)$-curve and the proper transforms of
all irreducible components of $\mathcal{F}$, we obtain a
$\mathbb{Q}$-homology projective plane $\bar{S}$ with $6$ quotient
singular points, a contradiction to \cite{HK1}.

We have proved that $E$ does not meet any end component of
$f^{-1}(p_i)$ for $1\le i\le 3$.
 This implies that
$E.f^{-1}(p_1)=E.f^{-1}(p_2)=0$ and $E.f^{-1}(p_3)=0$ if
$f^{-1}(p_3)$ has at most 2 components. We will show that
$E.f^{-1}(p_3)=0$ even if $f^{-1}(p_3)$ has more than 2
components, i.e., $p_3$ is of type $A_4$. Suppose that $p_3$ is of
type $A_4$ and $F_1, F_2, F_3, F_4$ be its 4 components whose dual
graph is $F_1\--F_2\--F_3\--F_4$.

 If $E$ meets $F_2$ at two distinct points, then
blowing up one of the two intersection points of $E$ and $F_2$
once, then contracting the proper transform of $E$ and the proper
transforms of all irreducible components of $\mathcal{F}$, we
obtain a $\mathbb{Q}$-homology projective plane $\bar{S}$ with one
noncyclic quotient singularity of type

\bigskip
$$
\begin{picture}(120,30)
\put(-100,15){$<3;2,1;2,1;3,2> :=  $}
\put(5,25){$\overset{-2}{\circ}-\overset{-3}{\circ}-\overset{-2}{\circ}-\overset{-2}{\circ}$}
 \put(33,15){\line(0,0){6}}
 \put(28,5){$\underset{-2}{\circ}$}
\end{picture}
$$

\bigskip\noindent
and $3$ cyclic singular points of order 2, 3, $q$ (see
\cite{Brieskorn} or Table 1 of \cite{HK1} for the notation of dual
graphs of noncyclic singularities). This surface has
$$e_{orb}=-1+\frac{1}{2}+\frac{1}{3}+\frac{1}{q}+\frac{1}{48}<0,$$
which violates the orbifold Bogomolov-Miyaoka-Yau inequality.

 If
$EF_2=EF_3=1$ and $E\cap F_2\cap F_3=\emptyset$, then blowing up
the intersection point of $E$ and $F_3$ once, then contracting the
proper transform of $E$ and the proper transforms of all
irreducible components of $\mathcal{F}$, we obtain a
$\mathbb{Q}$-homology projective plane $\bar{S}$ with one
noncyclic quotient singularity of type

\bigskip
$$
\begin{picture}(120,30)
\put(-100,15){$<2;2,1;2,1;5,2> :=  $}
\put(5,25){$\overset{-2}{\circ}-\overset{-2}{\circ}-\overset{-3}{\circ}-\overset{-2}{\circ}$}
 \put(33,15){\line(0,0){6}}
 \put(28,5){$\underset{-2}{\circ}$}
\end{picture}
$$

\bigskip\noindent and $3$ cyclic singular
points of order 2, 3, $q$. This surface has
$$e_{orb}=-1+\frac{1}{2}+\frac{1}{3}+\frac{1}{q}+\frac{1}{60}<0,$$
which also violates the orbifold Bogomolov-Miyaoka-Yau inequality.

 If
$EF_2=EF_3=1$ and $E\cap F_2\cap F_3\neq\emptyset$, then blowing
up the intersection point once, then contracting the proper
transform of $E$ and the proper transforms of all irreducible
components of $\mathcal{F}$, we obtain a $\mathbb{Q}$-homology
projective plane $\bar{S}$ with 6 quotient singular points, a
contradiction to \cite{HK1}.

 If
$EF_2=1$ and $EF=1$ for some component $F$ of $f^{-1}(p_i)$ for
some $i\neq 3$, then blowing up the intersection point of $E$ and
$F$ four times, then contracting the proper transform of $E$ with
a string of three $(-2)$-curves and the proper transforms of all
irreducible components of $\mathcal{F}$, we obtain a
$\mathbb{Q}$-homology projective plane $\bar{S}$ with one
noncyclic quotient singularity of type $E_8=<2;2,1;3,2;5,4>$

\bigskip
$$
\begin{picture}(120,30)
\put(-100,15){$<2;2,1;3,2;5,4>:= $}
\put(5,25){$\overset{-2}{\circ}-\overset{-2}{\circ}-\overset{-2}{\circ}-\overset{-2}{\circ}-\overset{-2}{\circ}-\overset{-2}{\circ}-\overset{-2}{\circ}$}
 \put(55,15){\line(0,0){6}}
 \put(50,5){$\underset{-2}{\circ}$}
\end{picture}
$$

\bigskip\noindent and $3$ cyclic singular points of order $\ge 2$, $\ge 3$, $\ge q$.
This surface has $$e_{orb}\le
-1+\frac{1}{2}+\frac{1}{3}+\frac{1}{q}+\frac{1}{120}<0,$$ which
violates the orbifold Bogomolov-Miyaoka-Yau inequality.\\ This
completes the proof of $E.f^{-1}(p_3)=0$. Thus $E.f^{-1}(p_4)=2$.
\end{proof}

In the following lemma, we do not assume that $H_1(S^0,
\mathbb{Z}) = 0$.

\begin{lemma}\label{q20} Let $S$ be a $\mathbb{Q}$-homology projective plane with exactly $4$
cyclic singular points $p_1, p_2, p_3, p_4$ of orders $(2,3,5,q)$.
$($We do not assume that $\gcd(q, 30) = 1.)$ Assume that $K_S$ is
ample. Assume that the order $3$ singularity is of type
$\frac{1}{3}(1,1)$. Then the following hold true.
\begin{enumerate}
\item $L\geq 12$ except possibly four cases, No.$1-4$ in Table
\ref{L11}. In each of these four cases, $S$ is rational and
$L=11$.
\item $q\geq 20$ except possibly one case, No.$1$ in Table \ref{L11}.
\end{enumerate}
\end{lemma}

\begin{proof}
(1) We have to consider the following types.
\bigskip
\begin{itemize} \item $A_1 + \frac{1}{3}(1,1) + A_4 +
\frac{1}{q}(1,q_1)$ \item $A_1 + \frac{1}{3}(1,1) +
\frac{1}{5}(1,2) + \frac{1}{q}(1,q_1)$ \item $A_1 +
\frac{1}{3}(1,1) + \frac{1}{5}(1,1) + \frac{1}{q}(1, q_1)$
\end{itemize}
\bigskip
Let $[n_1, \ldots,n_l]$ be the Hirzebruch-Jung continued fraction
corresponding to the singularity $p_4$. Since $K_S$ is ample,
Theorem \ref{bmy} implies that
$$0<K_{S'}^2-\mathcal{D}_{p_2}^2-\mathcal{D}_{p_3}^2-\mathcal{D}_{p_4}^2=K_S^2\leq 3e_{orb}(S)=\frac{1}{10}+\frac{3}{q}.$$
Since $K_{S'}^2=9-L$, $\mathcal{D}_{p_2}^2=-\frac{1}{3}$, Lemma
\ref{Dp} implies that
$$ L-7
+2l-\frac{1}{3}+\mathcal{D}_{p_3}^2-\frac{q_1+q_l+2}{q}<\sum
n_j\leq L-7
+2l-\frac{1}{3}+\mathcal{D}_{p_3}^2-\frac{q_1+q_l-1}{q}+\frac{1}{10}.$$
In particular, if $L$ is bounded, so is the number of possible
cases for $[n_1, \ldots,n_l]$.

Assume that  $L\leq 11$.

If $p_3$ is of type $A_4$, then $L=l+6$, $\mathcal{D}_{p_3}^2=0$
and the above inequality shows that $\sum n_j=3l-2$ or $3l-3$, so
up to permutation of $n_1,\ldots, n_l$,
$$\begin{array}{lll} [n_1,\ldots,
n_l]&=&[5,2,2,2,2], [4,3,2,2,2], [3,3,3,2,2];\\
&& [4,2,2,2,2], [3,3,2,2,2];\\
&& [4,2,2,2], [3,3,2,2];\\&& [3,2,2,2];\\ && [3,2,2];\\&&
[2,2,2];\\&& [2,2],\end{array}$$ hence there are 42 possible cases
for $[n_1, \ldots,n_l]$. Here we identify $[n_1, \ldots,n_l]$ with
its reverse $[n_l, \ldots,n_1]$.

If $p_3$ is of type $\frac{1}{5}(1,2)$, then $L=l+4$,
$\mathcal{D}_{p_3}^2=-\frac{2}{5}$ and $\sum n_j=3l-4$ or $3l-5$,
so up to permutation of $n_1,\ldots, n_l$,
$$\begin{array}{lll} [n_1,\ldots,
n_l]&=&[5,2,2,2,2,2,2], [4,3,2,2,2,2,2], [3,3,3,2,2,2,2];\\
&& [4,2,2,2,2,2,2], [3,3,2,2,2,2,2];\\
&& [4,2,2,2,2,2], [3,3,2,2,2,2];\\&& [3,2,2,2,2,2];\\ &&
[3,2,2,2,2];\\&& [2,2,2,2,2];\\&& [2,2,2,2],\end{array}$$ hence
there are 80 possible cases for $[n_1, \ldots,n_l]$ if $l\leq 7$.

If $p_3$ is of type $\frac{1}{5}(1,1)$, then $L=l+3$,
$\mathcal{D}_{p_3}^2=-\frac{9}{5}$ and $\sum n_j=3l-7$ or $3l-8$,
so up to permutation of $n_1,\ldots, n_l$,
$$\begin{array}{lll} [n_1,\ldots,
n_l]&=& [3,2,2,2,2,2,2,2], [2,2,2,2,2,2,2,2];\\
&& [2,2,2,2,2,2,2], \end{array}$$ hence there are 6 possible cases
for $[n_1, \ldots,n_l]$ if $l\leq 8$.

Among these $42+80+6=128$ cases, a direct calculation of
$D=|\det(R)|K_S^2$ shows that only 11  cases satisfy the condition
that $D$ is a positive square number (see Lemma \ref{coprime}(5)).
Table \ref{L11} describes the 11 cases.
\begin{table}[ht]
\caption{}\label{L11}
\renewcommand\arraystretch{1.5}
\noindent\[
\begin{array}{|c|l|c|c|c|c|}
\hline
{\rm No.} & \textrm{ Type\,\,of\,\,$R$} & q & K^2_S  & & 3e_{orb}\\
 \hline
1 & A_1+\frac{1}{3}(1,1)+\frac{1}{5}(1,1)+[2,2,2,2,2,2,2,2] & 9 & \frac{2}{15} & < & \frac{13}{30} \\ \hline
2 &     A_1+\frac{1}{3}(1,1)+\frac{1}{5}(1,2)+[4,2,2,2,2,2,2] & 22 & \frac{1}{165} & < &  \frac{13}{55}\\   \hline
3 &   A_1+\frac{1}{3}(1,1)+\frac{1}{5}(1,2)+[3,3,2,2,2,2,2] & 33 & \frac{2}{55} & < &  \frac{21}{110}\\   \hline
4 &   A_1+\frac{1}{3}(1,1)+\frac{1}{5}(1,2)+[3,2,2,3,2,2,2] & 43 & \frac{8}{645} & < &  \frac{73}{430}\\   \hline
5 &   A_1+\frac{1}{3}(1,1)+\frac{1}{5}(1,2)+[2,2,2,4,2,2,2] & 40 & \frac{1}{3} & > & \frac{7}{40}\\ \hline
6 &     A_1+\frac{1}{3}(1,1)+\frac{1}{5}(1,2)+[3,3,3,2,2,2,2] & 73 & \frac{1058}{1095} & > & \frac{103}{730} \\   \hline
7 &  A_1+\frac{1}{3}(1,1)+\frac{1}{5}(1,2)+[2,3,4,2,2,2,2] & 70 & \frac{25}{21} & > & \frac{1}{7} \\   \hline
8  &  A_1+\frac{1}{3}(1,1)+\frac{1}{5}(1,2)+[2,3,3,3,2,2,2] & 97 & \frac{1682}{1455} & > &  \frac{127}{970}\\   \hline
9 &  A_1+\frac{1}{3}(1,1)+\frac{1}{5}(1,2)+[2,2,4,3,2,2,2] & 78 & \frac{81}{65} & > &  \frac{9}{65}\\   \hline
10 &  A_1+\frac{1}{3}(1,1)+\frac{1}{5}(1,2)+[3,3,2,2,3,2,2] & 87 & \frac{128}{145} & > &  \frac{39}{290}\\   \hline
11 &  A_1+\frac{1}{3}(1,1)+\frac{1}{5}(1,2)+[2,3,3,2,2,3,2] & 103 & \frac{1568}{1545} & > &  \frac{133}{1030}\\
  \hline
\end{array}
\]
\end{table}

Among the 11 cases, only the first 4 cases satisfy the orbifold
Bogomolov-Miyaoka-Yau inequality $K_S^2\le 3e_{orb}$.

As for the first 4 cases of Table \ref{L11}, one can check that
none of them can be ruled out by any further lattice theoretic
argument, i.e., in each case the lattice $R$ can be embedded into
an odd unimodular lattice of signature $(1,L)$. This can be
checked by the local-global principle and the computation of
$\epsilon$-invariants (see e.g., \cite{HK1} Section 6).

To prove the rationality in each of the first 4 cases of Table
\ref{L11}, we will use the formulae from Proposition \ref{int}.
First note that $L = 11$ in each of the first 4 cases of Table
\ref{L11}.

\medskip
Case 1. Suppose that this case occurs on $S$ which is not
rational.\\ Note that $D=36$. Since ${\rm disc}(\bar{R})$ is a
cyclic group (Lemma \ref{general}), we see that
$\det(\bar{R})=\frac{\det(R)}{3^2}$, and hence
$D'=\frac{D}{3^2}=4$. By Lemma \ref{bound}, $S'$ contains a
$(-1)$-curve $E$ with $0< m\leq \frac{\sqrt{D'}}{L-9}=1$, i.e.,
$m=1$. By Proposition \ref{int}(1), we obtain
$$\underset{p}{\sum}\underset{j}{\sum} \Big(1-\frac{v_{j,p} + u_{j,p}}{q_p}  \Big)(EA_{j,p}) = 1 +
\frac{m}{\sqrt{D'}} K^2_S= \frac{16}{15}.$$ Looking at Table
\ref{finite2-1}, we see that there are non-negative integers $x,
y$ such that
$$\frac{x}{3}+\frac{3y}{5}=\frac{16}{15}.$$ It is easy to check that the equation has no
solution.
\begin{table}[ht]
\caption{}\label{finite2-1}
\renewcommand\arraystretch{1.5}
\noindent\[
\begin{array}{|c|c|c|c|c|c|c|c|c|c|c|c|}
\hline
 & [2] & [3]  & [5] & \multicolumn{8}{|c|}{[2,2,2,2,2,2,2,2]}  \\   \hline
j&1&1&1 &  1 &2&3&4&5&6&7&8\\
 \hline
1-\frac{v_j + u_{j}}{q} & 0 &\frac{1}{3}&\frac{3}{5}& 0&0&0&0&0&0&0&0\\
 \hline
\end{array}
\]
\end{table}

\medskip
Case 2. Suppose that this case occurs on $S$ which is not
rational.\\ Note that $D=4$. Since ${\rm disc}(\bar{R})$ is a
cyclic group (Lemma \ref{general}), we see that
$D'=\frac{D}{2^2}=1$. By Lemma \ref{bound}, $S'$ contains a
$(-1)$-curve $E$ with $0< m\leq
\frac{\sqrt{D'}}{L-9}=\frac{1}{2}$, a contradiction.

\medskip
Case 3. Suppose that this case occurs on $S$ which is not
rational.\\ Note that $D=36$. Since ${\rm disc}(\bar{R})$ is a
cyclic group (Lemma \ref{general}), we see that
$D'=\frac{D}{3^2}=4$.  By Lemma \ref{bound}, $S'$ contains a
$(-1)$-curve $E$ with $0< m\leq \frac{\sqrt{D'}}{L-9}=1$, i.e.,
$m=1$. By Proposition \ref{int} (1), we obtain
$$\underset{p}{\sum}\underset{j}{\sum} \Big(1-\frac{v_{j,p} + u_{j,p}}{q_p}  \Big)(EA_{j,p}) = 1 +
\frac{m}{\sqrt{D'}} K^2_S= \frac{56}{55}.$$
 Looking at Table \ref{finite2-3},
\begin{table}[ht]
\caption{}\label{finite2-3}
\renewcommand\arraystretch{1.5}
\noindent\[
\begin{array}{|c|c|c|c|c|c|c|c|c|c|c|c|}
\hline
 & [2] & [3]  & \multicolumn{2}{|c|}{[2,3]} & \multicolumn{7}{|c|}{[3,3,2,2,2,2,2]}  \\   \hline
j&1&1&1&2&  1 &2&3&4&5&6&7\\
 \hline
1-\frac{v_j + u_{j}}{q} & 0 &\frac{1}{3}&\frac{1}{5}& \frac{2}{5}
&
\frac{19}{33}& \frac{24}{33}&\frac{20}{33}&\frac{16}{33}&\frac{12}{33}&\frac{8}{33}&\frac{4}{33}\\
 \hline
 \frac{v_j u_{j}}{q} & \frac{1}{2} &\frac{1}{3}&\frac{3}{5}& \frac{2}{5}
&
\frac{13}{33}& \frac{18}{33}&\frac{40}{33}&\frac{52}{33}&\frac{54}{33}&\frac{46}{33}&\frac{28}{33}\\
 \hline
\end{array}
\]
\end{table}
we see that there are non-negative integers $x, y, z$ such that
$$\frac{x}{3}+\frac{y}{5}+\frac{z}{33}=\frac{56}{55}.$$ The equation has 3  solutions $(x, y, z) = (0,1,27), (1,1,16),
(2,1,5)$. Again by Table \ref{finite2-3}, we can rule out the
third solution. By Proposition \ref{int}(2), we obtain
$$\underset{p}{\sum}\underset{j}{\sum} \frac{v_j u_j}{q}(EA_{j})^2
\leq 1+\frac{m^2}{D'}K^2_S=\frac{111}{110},$$ which rules out the
first two solutions.

\medskip
Case 4. Suppose that this case occurs on $S$ which is not
rational.\\ Note that $D=4^2$.  Since the orders are pairwise
relatively prime, $D'=D$. By Lemma \ref{bound}, $S'$ contains a
$(-1)$-curve $E$ with $0< m\leq \frac{\sqrt{D}}{L-9}=2$, i.e.,
$m=1$ or 2. By Proposition \ref{int}, we obtain
$$\underset{p}{\sum}\underset{j}{\sum} \Big(1-\frac{v_{j,p} + u_{j,p}}{q_p}  \Big)(EA_{j,p}) = 1 +
\frac{m}{\sqrt{D}} K^2_S= \frac{647}{645}\quad {\rm or}\quad
\frac{649}{645}.$$ Looking at Table \ref{finite2-4},
\begin{table}[ht]
\caption{}\label{finite2-4}
\renewcommand\arraystretch{1.5}
\noindent\[
\begin{array}{|c|c|c|c|c|c|c|c|c|c|c|c|}
\hline
 & [2] & [3]  & \multicolumn{2}{|c|}{[2,3]} & \multicolumn{7}{|c|}{[3,2,2,3,2,2,2]}  \\   \hline
j&1&1&1&2&  1 &2&3&4&5&6&7\\
 \hline
1-\frac{v_j + u_{j}}{q} & 0 &\frac{1}{3}&\frac{1}{5}& \frac{2}{5}
&
\frac{23}{43}& \frac{26}{43}&\frac{29}{43}&\frac{32}{43}&\frac{24}{43}&\frac{16}{43}&\frac{8}{43}\\
 \hline
\end{array}
\]
\end{table}
we see that there are non-negative integers $x, y, z$ such that
$$\frac{x}{3}+\frac{y}{5}+\frac{z}{43}=\frac{647}{645}\quad {\rm or}\quad
\frac{649}{645}.$$ But it is easy to check that both equations
have no solution.

\medskip
(2) Suppose $2\le q\le 19$. A direct calculation shows that only 6
cases satisfy the condition that $D$ is a positive square number.
Table \ref{Lsix} contains the 6 cases.
\begin{table}[ht]
\caption{}\label{Lsix}
\renewcommand\arraystretch{1.5}
\noindent\[
\begin{array}{|l|c|c|c|c|}
\hline
\textrm{ Type\,\,of\,\,$R$} & q & K^2_S  & & 3e_{orb}\\
 \hline
A_1+\frac{1}{3}(1,1)+\frac{1}{5}(1,1)+A_8 & 9 & \frac{2}{15} & < &
\frac{13}{30} \\ \hline
A_1 + \frac{1}{3}(1,1)
+ A_4 + \frac{1}{4}(1,1) & 4 & \frac{10}{3} & > & \frac{17}{20}\\
\hline A_1 + \frac{1}{3}(1,1) + A_4 + \frac{1}{5}(1,2) & 5 &
\frac{26}{15} &
> &  \frac{7}{10}\\ \hline   A_1 +
\frac{1}{3}(1,1) + A_4
+ \frac{1}{6}(1,1) & 6 & 5 & > & \frac{3}{5}\\
 \hline A_1 + \frac{1}{3}(1,1) + \frac{1}{5}(1,1) +
 \frac{1}{6}(1,1) & 6 &
\frac{49}{5} & > & \frac{3}{5}\\ \hline A_1 + \frac{1}{3}(1,1) + A_4
+ \frac{1}{16}(1,3) & 16 &
\frac{10}{3} & > & \frac{23}{80} \\
  \hline
\end{array}
\]
\end{table}
 Among the 6 cases, only the first case satisfies the orbifold
Bogomolov-Miyaoka-Yau inequality $K_S^2\le 3e_{orb}$. But it is
already considered in (1).
\end{proof}

\begin{lemma}\label{blow2}
Let $S$ be a $\mathbb{Q}$-homology projective plane with exactly
$4$ cyclic singular points $p_1, p_2, p_3, p_4$ of orders
$(2,3,7,q)$, $11 \leq q \leq 41$, or $(2,3,11,13)$. Regard
$\mathcal{F}:=f^{-1}(Sing(S))$ as a reduced integral divisor on
$S'$. Assume that $S'$ contains a $(-1)$-curve $E$. Then,
$$E.\mathcal{F} \geq 2.$$ Moreover, if $E.\mathcal{F}=2$, then $E$ does not meet an end component of
 $f^{-1}(p_i)$ for any $i=1,2,3,4$.
\end{lemma}

\begin{proof} The proof of the first assertion is the same as that of Lemma \ref{blow}.

\medskip
To prove the second assertion, assume that $E.\mathcal{F} = 2$.\\
Suppose that $E$ meets an end component $F$ of $f^{-1}(p_i)$ for
some $1\le i\le 4$.

If $EF=1$, then $EF'=1$ for some other component $F'$ of
$f^{-1}(p_j)$, where $j$ may or may not be $i$. Assume that $E\cap
F\cap F'=\emptyset$. Blowing up the intersection point of $E$ and
$F'$ sufficiently many times, then contracting the proper
transform of $E$ with a string of $(-2)$-curves and the proper
transforms of all irreducible components of $\mathcal{F}$, we
obtain a $\mathbb{Q}$-homology projective plane $\bar{S}$ with $4$
quotient singular points such that $e_{orb}<0$ (see Lemma
\ref{cf}(6)), which violates the orbifold Bogomolov-Miyaoka-Yau
inequality. Assume that $E\cap F\cap F'\neq\emptyset$. Blowing up
the intersection point once, then contracting the proper transform
of $E$ and the proper transforms of all irreducible components of
$\mathcal{F}$, we obtain a $\mathbb{Q}$-homology projective plane
$\bar{S}$ with $6$ quotient singular points, a contradiction to
\cite{HK1}.

If $E$ intersects $F$ at 2 distinct points, then we get a similar
contradiction: blowing up one of the two intersection points of
$E$ and $F$ sufficiently many times, then contracting the proper
transform of $E$ with the adjacent string of $(-2)$-curves and the
proper transforms of all irreducible components of $\mathcal{F}$,
to obtain a $\mathbb{Q}$-homology projective plane $\bar{S}$ with
$4$ quotient singular points such that $e_{orb}<0$.

If $E$ intersects $F$ at 1 point with multiplicity 2, then blowing
up the intersection point twice and then contracting the proper
transform of $E$ with a $(-2)$-curve and the proper transforms of
all irreducible components of $\mathcal{F}$, we obtain a
$\mathbb{Q}$-homology projective plane $\bar{S}$ with $6$ quotient
singular points, a contradiction to \cite{HK1}.

In all cases, we get a contradiction. This proves the second
assertion.
\end{proof}

\section{Proof of Theorem 1.4}
Let $S$ be a $\mathbb{Q}$-homology projective plane with cyclic
quotient singularities such that

\begin{itemize}
\item $H_1(S^0, \mathbb{Z}) = 0$, \item
$S$ is not rational.
\end{itemize}

Assume that $|Sing(S)|=4$. In the previous section, we have
enumerated all possible $4$-tuples of orders of local fundamental
groups:

\bigskip
\begin{enumerate}
\item $(2,3,5,q)$, $q \geq 7$, $\gcd(q, 30) = 1$, \item
$(2,3,7,q)$, $11 \leq q \leq 41$, $\gcd(q, 42) = 1$, \item
$(2,3,11,13)$.
\end{enumerate}
\bigskip

For (2) and (3), we have seen that there are 24  different
possible types for $R$, the sublattice of $H^2(S',
\mathbb{Z})_{free}$ generated by all exceptional curves of the
minimal resolution $f:S'\to S$, as shown in Table \ref{finite0}.
Lemma \ref{24} rules out all these 24 cases, since we assume that
$S$ is not rational.

For (1), the order 3 singularity is of type $\frac{1}{3}(1,1)$
(Lemma \ref{noA2}), so it remains to consider the following cases:

\bigskip
\begin{itemize}
\item $A_1 + \frac{1}{3}(1,1) + A_4 + \frac{1}{q}(1,q_1)$, $q\geq
7$, $\gcd(q, 30) = 1$; \item $A_1 + \frac{1}{3}(1,1) +
\frac{1}{5}(1,2) + \frac{1}{q}(1,q_1)$, $q\geq 7$, $\gcd(q,30) =
1$;
\item $A_1 + \frac{1}{3}(1,1) + \frac{1}{5}(1,1) +
\frac{1}{q}(1,q_1)$, $q\geq 7$, $\gcd(q, 30) = 1$.
\end{itemize}

\bigskip\noindent
Since $S$ is not rational, $K_S$ is ample by Lemma
\ref{coprime}(4).\\
By Lemma \ref{q20} we may also assume that
\begin{itemize}
\item  $q\geq 20$ and $L\geq 12$.
\end{itemize}

\bigskip
We will show that none of the above cases occurs.  In the
following proof we do not assume that $\gcd(q, 30) = 1$ (so do not
assume that $H_1(S^0, \mathbb{Z}) = 0$), but still assume that
$K_S$ is ample. That is, in the following proof we will consider
the cases

\bigskip
\begin{itemize}
\item $A_1 + \frac{1}{3}(1,1) + A_4 + \frac{1}{q}(1,q_1)$, $q\geq 20$ and $L\geq 12$; \item $A_1 + \frac{1}{3}(1,1) +
\frac{1}{5}(1,2) + \frac{1}{q}(1,q_1)$, $q\geq 20$ and $L\geq 12$;
\item $A_1 + \frac{1}{3}(1,1) + \frac{1}{5}(1,1) +
\frac{1}{q}(1,q_1)$, $q\geq 20$ and $L\geq 12$
\end{itemize}

\bigskip\noindent
with the assumption that
\begin{itemize}
\item $K_S$ is ample and
$S$ is not rational.
\end{itemize}

\bigskip We will show that none of these cases occurs. The reason why we consider the situation without the
assumption that $\gcd(q, 30) = 1$ is that some part of the proof
below uses induction on $L={\rm rank}(R)$. After blowing down a
suitable $(-1)$-curve $E$ on $S'$, $$S'\to S'',$$ we contract
Hirzebruch-Jung chains of rational curves, $$S''\to \bar{S},$$ to
get a new $\mathbb{Q}$-homology projective plane $\bar{S}$ with
$L_{\bar{S}}=L-1$ having cyclic quotient singularities whose
orders may not be pairwise relatively prime.

By Lemma \ref{bound}, there is a $(-1)$-curve $E$ on $S'$ of the
form \eqref{E} with $$0<\frac{m}{\sqrt{D'}}\le \frac{1}{L-9}\le
\frac{1}{3}.$$ We will show that the existence of such a curve $E$
leads to a contradiction.

\medskip
{\bf Step 1.} \begin{enumerate} \item $K^2_S \leq \frac{1}{4}$,
\item $\frac{m}{\sqrt{D'}} K^2_S \leq \frac{1}{12}$, \item $\frac{m^2}{D'} K^2_S \leq \frac{1}{36}.$
\end{enumerate}

\begin{proof} Since $q\ge 20$, $$3e_{orb}(S)=\frac{1}{10}+\frac{3}{q}\leq \frac{1}{10}+\frac{3}{20}=\frac{1}{4}.$$
Since $K_S$ is ample, (1) follows from the orbifold BMY
inequality. (2) and (3) follow from (1) and the inequality
$\frac{m}{\sqrt{D'}}\le \frac{1}{3}$.
\end{proof}

Let $p_1,p_2,p_3,p_4$ be the 4 singular points. Assume that the
singularity $p_4$ is of type $[n_1,\ldots,n_l]$.
 Since $L\ge 12$, we see that $l\ge 6$.

\medskip
{\bf Step 2.} $E.f^{-1}(p_4)=2$ and $E.f^{-1}(p_i)=0$ for
$i=1,2,3$.

\begin{proof}
By Proposition \ref{int}(1)
$$\underset{p}{\sum}  \overset{
l_p}{\underset{j = 1}{\sum}}(1-\dfrac{v_{j,p}+ u_{j,p}}{q_p})
(EA_{j,p})= 1+\dfrac{m}{\sqrt{D'}} K^2_S.$$ By Lemma \ref{cf} we
see that $1-\dfrac{v_{j,p}+ u_{j,p}}{q_p}\ge 0$ for all $j, p$, so
by looking at only the terms with $p=p_4$, we get
$$\begin{array}{lll} E.f^{-1}(p_4)-\overset{
l}{\underset{j = 1}{\sum}} (\dfrac{v_{j}+ u_{j}}{q})
(EA_{j})&=&\overset{ l}{\underset{j = 1}{\sum}} (1-\dfrac{v_{j}+
u_{j}}{q}) (EA_{j})\\ & \leq &1+\dfrac{m}{\sqrt{D'}}
K^2_S,\end{array}$$ where $A_j:=A_{j,p_4}$, $v_j:=v_{j,p_4}$,
$u_j:=u_{j,p_4}$. By Proposition \ref{int}(2)
$$ \overset{
l}{\underset{j = 1}{\sum}} \dfrac{v_{j} u_{j}}{q} (EA_{j})^2\leq
1+\dfrac{m^2}{D'} K^2_S. $$ Adding these two inequalities side by
side, we get
$$E.f^{-1}(p_4)-\overset{ l}{\underset{j = 1}{\sum}}
(\dfrac{v_{j}+ u_{j}}{q}) (EA_{j})+\overset{ l}{\underset{j =
1}{\sum}} \dfrac{v_{j} u_{j}}{q} (EA_{j})^2\leq
2+\dfrac{m}{\sqrt{D'}} K^2_S+\dfrac{m^2}{D'} K^2_S. $$ By Lemma
\ref{uv}, $$\overset{ l}{\underset{j = 1}{\sum}} (\dfrac{v_{j}+
u_{j}}{q}) (EA_{j})\leq \overset{ l}{\underset{j = 1}{\sum}}
\dfrac{v_{j} u_{j}}{q} (EA_{j})^2 +\dfrac{2}{q}.$$ Thus
$$E.f^{-1}(p_4)\leq
2+\dfrac{m}{\sqrt{D'}} K^2_S+\dfrac{m^2}{D'} K^2_S+\dfrac{2}{q}<3,
$$ proving that $E.f^{-1}(p_4)\leq 2$.

\medskip
Assume that  $E.f^{-1}(p_4)=2$. By Proposition \ref{int}(1),(2)
$$\underset{p\neq p_4}{\sum}  \overset{
l_p}{\underset{j = 1}{\sum}}(1-\dfrac{v_{j,p}+ u_{j,p}}{q_p})
(EA_{j,p})= 1+\dfrac{m}{\sqrt{D'}} K^2_S-E.f^{-1}(p_4)+\overset{
l}{\underset{j = 1}{\sum}} (\dfrac{v_{j}+ u_{j}}{q}) (EA_{j}),$$
$$\underset{p\neq p_4}{\sum}  \overset{
l_p}{\underset{j = 1}{\sum}}\dfrac{v_{j,p} u_{j,p}}{q_p}
(EA_{j,p})^2\le 1+\dfrac{m^2}{D'} K^2_S - \overset{ l}{\underset{j
= 1}{\sum}} \dfrac{v_{j} u_{j}}{q} (EA_{j})^2.$$ Adding these two
side by side, then using Lemma \ref{uv}, we get

$$\begin{array}{ll} &\underset{p\neq p_4}{\sum}  \overset{
l_p}{\underset{j = 1}{\sum}}\Big( (1-\dfrac{v_{j,p}+
u_{j,p}}{q_p}) (EA_{j,p})+\dfrac{v_{j,p} u_{j,p}}{q_p}
(EA_{j,p})^2\Big)\\
&\quad \le \dfrac{m}{\sqrt{D'}} K^2_S+\dfrac{m^2}{D'}
K^2_S+\overset{ l}{\underset{j = 1}{\sum}} (\dfrac{v_{j}+
u_{j}}{q}) (EA_{j}) -\overset{ l}{\underset{j = 1}{\sum}}
\dfrac{v_{j} u_{j}}{q} (EA_{j})^2\\ &\quad \le
\dfrac{m}{\sqrt{D'}} K^2_S+\dfrac{m^2}{D'} K^2_S+\dfrac{2}{q}\\
&\quad \le
\dfrac{1}{12}+\dfrac{1}{36}+\dfrac{2}{20}<\dfrac{1}{3}.\end{array}$$
Now from Table \ref{235} it is easy to see that $E.f^{-1}(p_i)=0$
for $i=1,2,3$.
\begin{table}[ht]
\caption{}\label{235}
\renewcommand\arraystretch{1.5}
\noindent\[
\begin{array}{|c|c|c|c|c|c|c|c|c|c|}
\hline
 & [2] & [3]  & [5] & \multicolumn{2}{|c|}{[3,2]} & \multicolumn{4}{|c|}{[2,2,2,2]}  \\   \hline
j&1&1&1&1&2&  1 &2&3&4\\
 \hline
1-\frac{v_j + u_{j}}{q} & 0 &\frac{1}{3} &\frac{3}{5}&\frac{2}{5}&
\frac{1}{5}
&0&0&0&0\\
 \hline
 \frac{v_j u_{j}}{q} & \frac{1}{2} &\frac{1}{3}& \frac{1}{5}&\frac{2}{5}&\frac{3}{5}
&\frac{4}{5}& \frac{6}{5}&\frac{6}{5}&\frac{4}{5}\\
 \hline
\end{array}
\]
\end{table}

\medskip
Assume that  $E.f^{-1}(p_4)=1$, i.e., $EA_s=1$ for some $s$ and $EA_j=0$ for all $j\neq s$.\\
Lemma \ref{uv} gives $$\overset{ l}{\underset{j = 1}{\sum}}
(\dfrac{v_{j}+ u_{j}}{q}) (EA_{j})\leq \overset{ l}{\underset{j =
1}{\sum}} \dfrac{v_{j} u_{j}}{q} (EA_{j})^2 +\dfrac{1}{q}.$$ Thus
$$\begin{array}{ll}&\underset{p\neq p_4}{\sum}  \overset{
l_p}{\underset{j = 1}{\sum}}\Big( (1-\dfrac{v_{j,p}+
u_{j,p}}{q_p}) (EA_{j,p})+\dfrac{v_{j,p} u_{j,p}}{q_p}
(EA_{j,p})^2\Big)\\ &\quad \le 1+\dfrac{m}{\sqrt{D'}}
K^2_S+\dfrac{m^2}{D'} K^2_S+\dfrac{1}{q}\\ &\quad \le 1+
\dfrac{1}{12}+\dfrac{1}{36}+\dfrac{1}{20}<\dfrac{7}{6}.\end{array}$$
Now Table \ref{235} easily  gives
$E.(f^{-1}(p_1)+f^{-1}(p_2)+f^{-1}(p_3))\leq 1$. But this
contradicts Lemma \ref{blow}.

\medskip
Assume that  $E.f^{-1}(p_4)=0$.\\ In this case, we have
$$\underset{p\neq p_4}{\sum} \overset{ l_p}{\underset{j =
1}{\sum}}(1-\dfrac{v_{j,p}+ u_{j,p}}{q_p}) (EA_{j,p})=
1+\dfrac{m}{\sqrt{D'}} K^2_S.$$ Since $0< \dfrac{m}{\sqrt{D'}}
K^2_S\le \frac{1}{12}$, we have $$1< \underset{p\neq p_4}{\sum}
\overset{ l_p}{\underset{j = 1}{\sum}}(1-\dfrac{v_{j,p}+
u_{j,p}}{q_p}) (EA_{j,p})\le 1+\frac{1}{12}.$$ It is easy to see
that Table \ref{235} contains no solution to this inequality.
\end{proof}

Now we have 4 cases
\begin{enumerate}
\item $E.f^{-1}(p_i)=0$ for $i=1,2,3$, and $E$ meets one component
of $f^{-1}(p_4)$ with multiplicity 2. \item $E.f^{-1}(p_i)=0$ for
$i=1,2,3$, and $E$ meets two non-end components of $f^{-1}(p_4)$.
 \item
$E.f^{-1}(p_i)=0$ for $i=1,2,3$, and $E$ meets both end components
of $f^{-1}(p_4)$.  \item $E.f^{-1}(p_i)=0$ for $i=1,2,3$, and $E$
meets an end component and a non-end component of $f^{-1}(p_4)$.
\end{enumerate}

\medskip
{\bf Step 3.} Case (1) cannot occur.

\begin{proof} Suppose that Case (1) occurs, i.e., $EA_s=2$ for some $1\le s\le
l$, $EA_j=0$ for $j\neq s$.

If $1<s<l$, then Proposition \ref{int}(1),(3) give
$$1-\dfrac{m}{\sqrt{D'}} K^2_S=2(\dfrac{v_{s}+ u_{s}}{q})$$ and $$
1+\dfrac{m^2}{D'} K^2_S= 4\dfrac{v_{s} u_{s}}{q}.$$ Subtracting
the first equality multiplied by $2$ from the  second, we get
$$\dfrac{m^2}{D'} K^2_S+2\dfrac{m}{\sqrt{D'}} K^2_S-1=
 4\dfrac{v_{s} u_{s}}{q}-4(\dfrac{v_{s}+ u_{s}}{q})>0. $$
On the other hand, by Step 1,
$$\dfrac{m^2}{D'} K^2_S+2\dfrac{m}{\sqrt{D'}} K^2_S-1
\le \dfrac{1}{36}+\dfrac{2}{12}-1<0,$$ a contradiction.

 If $s=1$, then Proposition \ref{int}(1),(3) give
$$1-\dfrac{m}{\sqrt{D'}} K^2_S=2(\dfrac{v_{1}+ 1}{q})$$ and $$1+\dfrac{m^2}{D'} K^2_S= 4\dfrac{v_{1}}{q}.$$ Eliminating $\frac{v_{1}}{q}$, we get
$$1=\dfrac{m^2}{D'} K^2_S+2\dfrac{m}{\sqrt{D'}} K^2_S+\dfrac{4}{q}\le \frac{1}{36}+\frac{2}{12}+\frac{4}{20}<1, $$
 a contradiction.
\end{proof}

\medskip
{\bf Step 4.} Case (2) cannot occur.

\begin{proof} Suppose that Case (2) occurs, i.e., $EA_s=EA_t=1$ for some
$1<s<t<l$, $EA_j=0$ for $j\neq s, t$. Proposition \ref{int}(1),(2)
give
$$1-\dfrac{m}{\sqrt{D'}} K^2_S =\dfrac{v_{s}+
u_{s}}{q}+\dfrac{v_{t}+ u_{t}}{q}$$ and  $$ 1+\dfrac{m^2}{D'}
K^2_S\ge \dfrac{v_{s} u_{s}}{q}+\dfrac{v_{t}u_{t}}{q}.$$
Subtracting the equality multiplied by $\frac{4}{3}$ from the
inequality, we get
$$1+\dfrac{m^2}{D'} K^2_S-\dfrac{4}{3}+\dfrac{4m}{3\sqrt{D'}} K^2_S\ge
\dfrac{v_{s} u_{s}}{q}+\dfrac{v_{t}u_{t}}{q}-
\dfrac{4}{3}(\dfrac{v_{s}+ u_{s}}{q}+\dfrac{v_{t}+ u_{t}}{q})\ge
0,
$$ where the last inequality follows from $$vu-
\dfrac{4}{3}(v+u)=(v-\dfrac{4}{3})(u-\dfrac{4}{3})-\dfrac{16}{9}\ge
0$$ for $v\ge 2, u\ge 2, v+u\ge 7$. ($l\ge 6$ implies $v+u\ge
7$.)\\
Since $\dfrac{m^2}{D'} K^2_S+\dfrac{4m}{3\sqrt{D'}}
K^2_S<\dfrac{1}{3}$, it gives a contradiction.
\end{proof}

\medskip
{\bf Step 5.} Case (3) cannot occur.

\begin{proof} Suppose that Case (3) occurs, i.e., $EA_1=EA_l=1$, $EA_j=0$
for $j\neq 1, l$. Then, by Proposition \ref{int} (1), we obtain
$$\frac{q_1+q_l+2}{q} = 1 -\frac{m}{\sqrt{D'}} K^2_S.$$
Also by Proposition \ref{int} (3), we obtain
$$\frac{q_1+q_l+2}{q} = 1 +  \frac{m^2}{D'}K^2_S.$$
From these two equations we obtain $m = -\sqrt{D'}$ and hence
$-K_S$ is ample by Lemma \ref{general}(5).
\end{proof}

\medskip
{\bf Step 6.} Case (4) cannot occur.

\begin{proof} Suppose that Case (4) occurs, i.e., $EA_1=EA_t=1$ for some $1<t<l$ and $EA_j=0$ for
$j\neq 1, t$. Proposition \ref{int}(1),(3) give
$$1-\dfrac{m}{\sqrt{D'}} K^2_S =\dfrac{q_1+1}{q}+\dfrac{v_{t}+
u_{t}}{q}=\dfrac{q_1-1}{q}+\dfrac{v_{t}+ (u_{t}+2)}{q}$$ and  $$
1+\dfrac{m^2}{D'}
K^2_S=\dfrac{q_1}{q}+\dfrac{v_{t}u_{t}}{q}+2\dfrac{v_{t}}{q}=\dfrac{q_1}{q}+\dfrac{v_{t}(u_{t}+2)}{q}.$$
Subtracting the first equality multiplied by $\frac{3}{2}$ from
the second, we get
$$\begin{array}{lll} 1+\dfrac{m^2}{D'} K^2_S-\dfrac{3}{2} +\dfrac{3m}{2\sqrt{D'}}
K^2_S&=&\dfrac{q_1}{q}-\dfrac{3(q_1-1)}{2q}
+\dfrac{v_{t}(u_{t}+2)}{q}- \dfrac{3}{2}\Big(\dfrac{v_{t}+
(u_{t}+2)}{q}\Big)\\ &>&
\dfrac{q_1}{q}-\dfrac{3(q_1-1)}{2q}=-\dfrac{q_1-3}{2q},
\end{array}$$ where the inequality follows from
$$vu'-
\dfrac{3}{2}(v+u')=(v-\dfrac{3}{2})(u'-\dfrac{3}{2})-\dfrac{9}{4}>
0$$ for $v\ge 2, u'\ge 4, v+u'\ge 9$. ($l\ge 6$ implies
$v+u'=v+(u+2)\ge 9$.) Thus
$$\dfrac{q_1}{2q}>\dfrac{q_1-3}{2q}>\dfrac{1}{2}-\dfrac{m^2}{D'} K^2_S-\dfrac{3m}{2\sqrt{D'}}
K^2_S\ge
\dfrac{1}{2}-\dfrac{1}{36}-\dfrac{3}{2}\cdot\dfrac{1}{12}=\dfrac{25}{72},$$
hence
$$ \dfrac{q_1}{q}> \dfrac{25}{36} > \dfrac{1}{2}.$$ It implies, in particular, that
$$n_1 = 2.$$

We claim that $n_t = 2$. Suppose that $n_t
> 2$. Let $$\sigma:S' \rightarrow S''$$ be the blow down of the (-1)-curve $E$,
and $$g:S'' \rightarrow \bar{S}$$ be the contraction to another
$\mathbb{Q}$-homology projective plane $\bar{S}$ with
$$L_{\bar{S}}:=b_2(S'')-1=L-1.$$ Note that $\bar{S}$ has 3
singularities $\bar{p}_1, \bar{p}_2, \bar{p}_3$ of order 2,3,5 of
the same type as $S$, and a singularity $\bar{p}_4$ of order $q'$
with $q'<q$. The latter follows from Lemma \ref{cf}(6). Moreover
the image $\bar{A}_1$ on $S''$ is a $(-1)$-curve, and the images
$\bar{A}_2, \ldots, \bar{A}_l$ are the components of
$g^{-1}(\bar{p}_4)$.

We claim that $K_{\bar{S}}$ is ample. To prove this, note first
that $K_{\bar{S}}$ is ample if and only if the coefficient of
$\bar{A}_1$ in $g^*K_{\bar{S}}$, when written as a linear
combination of $\bar{A}_1$ and $g$-exceptional curves, is
positive. Let $C$ be the coefficient. From the adjunction formula
$$K_{S'}=f^*K_S- \sum{\mathcal{D}_{p_i}}=\sigma^*(g^*K_{\bar{S}}- \sum{\mathcal{D}_{\bar{p}_i}})+E,$$
we see that $C$ is equal to the coefficient of $A_1$ in $K_{S'}$,
when written as a linear combination of $E$ and $f$-exceptional
curves. To compute $C$, we localize at $p_4$ and write
$$f^*K_S=xE+\sum(y_jA_j),$$  $$ \mathcal{D}_{p_4} =
\sum(d_jA_j)$$ for some rational numbers $x, y_j, d_j$. Then
$$C=y_1-d_1.$$
 Since $E$ is of the form \eqref{E}, it is easy to
see $$x=\frac{\sqrt{D'}}{m}.$$  From the two systems of equations
$$(f^*K_S)A_i=0,\,\,\,(1\le i\le l),$$ and
$$(\mathcal{D}_{p_4})A_i=-n_i+2,\,\,\,(1\le i\le l), $$ we get
$$y_1=\frac{x(q_1+v_t)}{q}, \quad
d_1=1-\frac{q_1+1}{q}$$ respectively. Now since $x\ge L-9\ge 3$
and $\frac{q_1}{q}> \frac{25}{36}$, we see that
$$C=y_1-d_1=\frac{x(q_1+v_t)}{q}+\frac{q_1+1}{q}-1\ge \frac{4q_1}{q}+\frac{3v_t+1}{q}-1>0.$$
This proves that $K_{\bar{S}}$ is ample. If $\bar{S}$ has
$L_{\bar{S}}<12$ or $q'<20$, then we are done by Lemma \ref{q20}.
Otherwise, we can find a $(-1)$-curve $E'$ on $S''$ of the form
\eqref{E} with $$0<\frac{m}{\sqrt{D'}}\le
\frac{1}{L_{\bar{S}}-9}\le \frac{1}{3}.$$ We restart with $E'$  on
$S''$ from Step 1. Then, by Step $1$ to Step $5$, we may assume
that $E'$ satisfies the case (4), i.e., we may assume that
$E'\bar{A}_2=E'\bar{A}_{t'}=1$ with $2<t'<l$. Here $\bar{A}_2,
\ldots, \bar{A}_l$ are the components lying over the singularity
$\bar{p}_4$. If $-\bar{A}_{t'}^2>2$, we repeat the above process.
Since each process decreases by 1 the number $L$, we may assume
that $n_t=2$ at certain stage. Now by Lemma \ref{cf}(3)
$$\frac{u_tv_t}{q} \geq \frac{1}{n_t}=\frac{1}{2}.$$ Thus
$$\frac{37}{36}\ge 1 + \frac{m^2}{D'}K^2_S =\frac{q_1}{q}+\frac{u_t v_t + 2v_t}{q}
> \frac{q_1}{q}+\frac{u_tv_t}{q} \geq \frac{25}{36}+ \frac{1}{2}=\frac{43}{36},$$
a contradiction.
\end{proof}

This completes the proof of  Theorem \ref{main}.

\section{Log del Pezzo surfaces of rank one}
Throughout this section, $S$ denotes a $\mathbb{Q}$-homology
projective plane with quotient singularities such that $-K_S$ is
ample, i.e., $S$ is a log del Pezzo surface of rank one. Let
$$f:S' \rightarrow S$$ be a minimal resolution of $S$. Let
$$\mathcal{F}:=f^{-1}(Sing(S))$$ be the reduced exceptional divisor of $f$.

We review the work of Zhang \cite{Zhang}, Gurjar and Zhang \cite{GZ} and Belousov
\cite{Belousov} on log del Pezzo surfaces of rank one.

\begin{lemma}[]\label{-k} $B^2\ge -1$ for any irreducible curve $B\subset S'$ not contracted by $f:S' \rightarrow S$.
\end{lemma}

\begin{proof} This is well-known (cf. \cite{HK2}, Lemma 2.1).
\end{proof}

\begin{theorem}[\cite{Belousov}]\label{dp-rational}  $S$ has at most $4$ singular points.
\end{theorem}

The following lemma is given in Lemma 4.1 in \cite{Zhang}, and can
also be easily derived from the inequality of Proposition
\ref{int}(1) when $S$ has only cyclic singularities.

\begin{lemma}[\cite{Zhang}]\label{dp3}
Let $E$ be a $(-1)$-curve on $S'$. Let $A_1, \ldots, A_r$ exhaust
all
 irreducible components of $\mathcal{F}$ such that $EA_i > 0$. Suppose
that $A^2_1 \geq A^2_2 \geq \ldots \geq A^2_r$. Then the $r$-tuple
 $(-A^2_1, \ldots, -A^2_r )$ is one of the following:
$$ (2, \ldots, 2, n), n \geq 2, \,\, (2, \ldots, 2, 3, 3), \,\, (2, \ldots, 2, 3, 4), \,\, (2, \ldots, 2, 3, 5).$$
\end{lemma}

An irreducible curve $C$ on $S'$ is called a {\it minimal curve}
if $C.(-f^*K_S)$ attains the minimal positive value.

\begin{lemma}[\cite{Zhang}]\label{dp4}
A minimal curve $C$ is a smooth rational curve.
\end{lemma}

\begin{lemma}[\cite{Zhang}, Lemma 2.1, \cite{GZ}, Remark 3.4]\label{dp5} Let $C$ be a minimal curve.
Suppose that $|C+\mathcal{F}+K_{S'}| \neq \emptyset$. Then there
is a unique decomposition $\mathcal{F} = \mathcal{F}' +
\mathcal{F}''$ such that
    \begin{enumerate}
    \item $\mathcal{F}'$ consists of $(-2)$-curves not meeting $C+\mathcal{F}''$,
    \item $C+\mathcal{F}''+K_{S'} \sim 0$,
    \item $\mathcal{F}''=f^{-1}(p)$ for some singular point $p$ unless $\mathcal{F}''=0.$
    \end{enumerate}
    Furthermore, if $\mathcal{F}''\neq 0$, then
$C\mathcal{F}'' = C\mathcal{F} = 2$ and one of the
following holds:
    \begin{enumerate}
    \item $\mathcal{F}''$ consists of one irreducible component, which $C$ meets in a single point with multiplicity 2 or in two points,
    \item $\mathcal{F}''$ consists of two irreducible components, whose intersection point $C$ passes through,
    \item $\mathcal{F}''$ consists of at least two irreducible components, and $C$ meets the two end components of $\mathcal{F}''$.
    \end{enumerate}
\end{lemma}

\begin{lemma}[\cite{GZ}, Proposition 3.6]\label{dp7}
Let $C$ be a minimal curve. Suppose that $|C+\mathcal{F}+K_{S'}| =
\emptyset$. Then $C$ is a $(-1)$-curve, or $S'=S\cong\mathbb{P}^2$ and $C$ is a line,
or $S$ is a Hirzebruch surface with the negative section
contracted and $C$ is a fibre on the Hirzebruch surface.
\end{lemma}

\begin{lemma}[\cite{Belousov}, Lemma 4.1]\label{dp8}
Suppose that $S'$ contains a minimal curve $C$ with $C^2 = -1$.
Suppose that $|C+\mathcal{F}+K_{S'}| = \emptyset$. Then
$C\mathcal{F}' \leq 1$ for any connected component $\mathcal{F}'$
of $\mathcal{F}$.
\end{lemma}

\begin{lemma}[\cite{Zhang}, Lemma 4.4]\label{dp9}
Suppose that $S'$ contains a minimal curve $C$ with $C^2 = -1$.
Suppose that $|C+\mathcal{F}+K_{S'}| = \emptyset$, and that $C$
meets exactly two components $F_1, F_2$ of $\mathcal{F}$. Then
either $F^2_1 = -2$ or $F^2_2 = -2$.
\end{lemma}

The following lemma was proved in (\cite{Zhang}, the proof of Lemma 5.3).

\begin{lemma} \label{dp9-1}
With the same assumption as in Lemma \ref{dp9}, assume further
that $F^2_1 = F^2_2 = -2$. If $F_1$ is not an end component, then
one of the following two cases holds:
    \begin{enumerate}
    \item There exists another minimal $(-1)$-curve $C'$ such
        that $|C'+\mathcal{F}+K_{S'}| \neq \emptyset$.
    \item $F_2 = f^{-1}(p_i)$ for some singular point $p_i$.
    \end{enumerate}
\end{lemma}

\begin{lemma}\label{dp10}
Suppose that $S'$ contains a minimal curve $C$ with $C^2 = -1$.
Suppose that $|C+\mathcal{F}+K_{S'}| = \emptyset$, and that $C$
meets three components $F_1, F_2, F_3$ of $\mathcal{F}$ and
possibly more. Define
$$G:=2C+F_1 + F_2 +F_3+ K_{S'}.$$ Then either $G\sim 0$ or $G \sim
\Gamma$ for some $(-1)$-curve $\Gamma$ such that $C\Gamma =
F_i\Gamma = 0$ for $i = 1,2,3$. Furthermore, the following hold
true.
\begin{enumerate}
    \item  In the first case, there are $3$ singular points $p_1,
p_2, p_3$ such that $f^{-1}(p_i)=F_i$,
    and $C$ meets no component of $\mathcal{F}-(F_1+F_2+F_3)$.
    \item In the second
    case,
    \begin{enumerate}
  \item  $L = 2 - (F^2_1 + F^2_2 + F^2_3)$, where $L$ is the number of irreducible components of
  $\mathcal{F}$,
\item each curve in $\mathcal{F}-F_1-F_2-F_3$ is a $(-2)$- or a $(-3)$-curve
 and there are at most two $(-3)$-curves in
$\mathcal{F}-F_1-F_2-F_3$,
\item each connected component of $\mathcal{F}$ contains at most one $(-n)$-curve with $n \geq 3$.
\end{enumerate}
    \end{enumerate}
\end{lemma}

\begin{proof}
 The main assertion is exactly (\cite{Zhang}, Lemma 2.3).

(1) Let $F_i$ be an irreducible component of $f^{-1}(p_i)$.
Suppose that $f^{-1}(p_i)$ has at least $2$ irreducible
components.
    Then there is an irreducible component $I$ of $f^{-1}(p_i)$ such that $IF_i = 1$. By Lemma
    \ref{dp8}, $IC=0$, hence
        $$0=IG = I.(2C+F_1 + F_2 + F_3+ K_{S'}) = IF_i+IK_{S'} =1 - I^2 - 2.$$
        Thus $I^2 = -1$, a contradiction.\\ Suppose that $C$ meets
        a component $J$ of $\mathcal{F}-(F_1 + F_2 + F_3)$. Then
$$0=JG = J.(2C+F_1 + F_2 + F_3+ K_{S'}) = 2+JK_{S'},$$
        so $J^2 = 0$, a contradiction.

(2-a) We note that
        $$G^2 = (2C+F_1 + F_2 + F_3+ K_{S'})^2 = 1 -L - (F^2_1 + F^2_2 + F^2_3).$$
        Since $G^2 =\Gamma^2= -1$, we have $L = 2 - (F^2_1 + F^2_2 + F^2_3).$

        (2-b) and (2-c) are exactly (\cite{GZ}, Lemma 6.6).
\end{proof}

The following lemma was proved  in (\cite{Zhang}, the proof of
Lemma 5.2).

\begin{lemma}\label{dp12}
With the same assumption as in Lemma \ref{dp10}, assume further
that $2C+F_1 + F_2 +F_3+ K_{S'}\sim \Gamma$ for some $(-1)$-curve
$\Gamma$, and that at least two of $F_1, F_2, F_3$ are
$(-2)$-curves. Then one of the following two cases holds:
    \begin{enumerate}
    \item There exists another minimal $(-1)$-curve $C'$ such that $|C'+\mathcal{F}+K_{S'}| \neq \emptyset$.
    \item $S$ has a non-cyclic singularity.
    \end{enumerate}
\end{lemma}

\section{Proof of Theorem \ref{dpmain}}

 Let $S$ be a $\mathbb{Q}$-homology
projective plane with cyclic quotient singularities such that

\begin{itemize}
\item $H_1(S^0, \mathbb{Z}) = 0$, \item
$-K_S$ is ample.
\end{itemize}

Assume that $S$ has exactly $4$ cyclic singularities $p_1$, $p_2$,
$p_3$, $p_4$. In Section 5, we have enumerated all possible
$4$-tuples of orders of local fundamental groups:

\bigskip
\begin{enumerate}
\item $(2,3,5,q)$, $q \geq 7$, $\gcd(q, 30) = 1$, \item
$(2,3,7,q)$, $11 \leq q \leq 41$, $\gcd(q, 42) = 1$, \item
$(2,3,11,13)$.
\end{enumerate}
\bigskip

For (2) and (3), we have seen that there are 24 different possible
types for $R$, the sublattice of $H^2(S', \mathbb{Z})$ generated
by all exceptional curves of the minimal resolution $f:S'\to S$,
as shown in Table \ref{finite0}.

For (1), the order 3 singularity is of type $\frac{1}{3}(1,1)$
(Lemma \ref{noA2}), so it remains to consider the following cases:

\bigskip
\begin{itemize}
\item $A_1 + \frac{1}{3}(1,1) + \frac{1}{5}(1,1)+ \frac{1}{q}(1,q_1)$, $q\geq
7$, $\gcd(q, 30) = 1$; \item $A_1 + \frac{1}{3}(1,1) +
\frac{1}{5}(1,2) + \frac{1}{q}(1,q_1)$, $q\geq 7$, $\gcd(q,30) =
1$;
\item $A_1 + \frac{1}{3}(1,1) + A_4  +
\frac{1}{q}(1,q_1)$, $q\geq 7$, $\gcd(q, 30) = 1$;
\item the 24 cases in Table \ref{finite0}.
\end{itemize}

\bigskip\noindent
Let
$$\mathcal{F}=f^{-1}(Sing(S))$$ be the reduced exceptional divisor of the minimal resolution $f:S' \rightarrow S$.

Let $C$ be a (fixed) minimal curve on $S'$.  Since $-K_S$ is
ample, by Lemma \ref{general}, $C$ can be written as
\begin{equation}\label{C}
C =  -mM + \underset{p \in Sing(S)}{\sum}
\overset{l_p}{\underset{i = 1}{\sum}} a_{i, p} A_{i,p}
\end{equation}
 for some integer $m > 0$ and some $a_{i,p}\in
 \frac{1}{c}\mathbb{Z}$.

\subsection{Step 1. $|C+\mathcal{F}+K_{S'}| =\emptyset$}
\begin{proof}
Suppose that $|C+\mathcal{F}+K_{S'}| \neq\emptyset$. By Lemma
\ref{dp5}, we see that $S$ has at least 3 rational double
points.

 In the case of
$(2,3,5,q)$, by Lemma \ref{noA2} we see that $S$ has 3 rational
double points, only if the singularities are of type
$A_1+[3]+A_4+A_{q-1}$. In this case, $$L=q+5\,\,\,{\rm and}\,\,\,
K_S^2=9-(q+5)+\frac{1}{3}<0,$$ a contradiction.

We also see that each of the 24 cases from Table 1 has at most 2 rational
double points.
\end{proof}

\subsection{Step 2.}
\begin{enumerate}
    \item $C$ is a $(-1)$-curve.
    \item $C\mathcal{F}=3$, and $C$ meets three distinct components $F_1, F_2, F_3$ of
$\mathcal{F}$.
\end{enumerate}

\begin{proof} (1) It immediately follows from Lemma \ref{dp7} since $S$ has $4$ singularities.

(2) By Lemma \ref{dp8}, $C\mathcal{F}\le 4$. Since $C^2=-1<0$ and
the lattice $R$ is negative definite, $C\mathcal{F} \geq 1$.

Assume that $C\mathcal{F} = 1$. Blowing up the intersection point,
then contracting the proper transform of $C$ and the proper
transforms of all irreducible components of $\mathcal{F}$, we
obtain a $\mathbb{Q}$-homology projective plane with $5$ quotient
singularities, which contradicts the result of \cite{HK1} since
$S$ is rational.

Assume that $C\mathcal{F} = 4$. By Lemma \ref{dp8}, $C$ meets four
components $F_1, F_2, F_3, F_4$ of $\mathcal{F}$, where
$F_i\subset f^{-1}(p_i)$. Then $G \sim \Gamma$ by Lemma \ref{dp10}
(1). By Lemma \ref{dp3}, at least two of $F_1, F_2, F_3, F_4$ have
self-intersection $-2$. Thus, by Lemma \ref{dp12}, there exists
another minimal $(-1)$-curve $C'$ such that
$|C'+\mathcal{F}+K_{S'}| \neq \emptyset$. This is impossible by
Step 1.

Assume that $C\mathcal{F} = 2$. \\
(a) Suppose that the case $(2,3,5,q)$ occurs for some $q \geq 7$
with $\gcd(q, 30) =1$.  By Lemma \ref{blow}, $C.f^{-1}(p_4)=2$.
But, By Lemma \ref{dp8}, $C.f^{-1}(p_4)\le 1$, a contradiction.  \\
(b) Now suppose that one of the $24$ cases of Table 1 occurs. By
Lemma \ref{dp8}, there are two components $F_1$ and $F_2$ of
$\mathcal{F}$ with $CF_1=CF_2=1$. By Lemma \ref{dp9}, we may
assume that $F^2_1 = -2$. Moreover, by Lemma \ref{blow2}, $C$ does
not meet an end component of $f^{-1}(p_i)$ for any $i$, i.e., both
$F_1$ and $F_2$ are middle components. Thus $F^2_2 \neq-2$ by
Lemma \ref{dp9-1} and Step 1. After contracting the $(-1)$-curve
$C$, by contracting the proper transforms of all irreducible
components of $\mathcal{F} - F_1$, we obtain a
$\mathbb{Q}$-homology projective plane with $5$ quotient
singularities, which contradicts the result of \cite{HK1} since
$S$ is rational.
\end{proof}

\subsection{Step 3.}
 $2C+F_1 + F_2 +F_3+ K_{S'} \sim \Gamma$ for
some $(-1)$-curve $\Gamma$.

\begin{proof}
Suppose that $$2C+F_1 + F_2 +F_3+ K_{S'} \sim 0.$$ Then, by Lemma
\ref{dp10}(1),
 each $F_i$ is equal to the inverse image of a singular point of $S$.
By Table 1 and Lemma \ref{noA2}, only the following cases satisfy
this condition:

\bigskip $$\begin{array}{ll} A_1+A_2+[7]+[13]\,\,& ({\rm Case}\,\, 1,\,\,{\rm
Table}\,\,1),\\ A_1 +[3]+[2,2,2,2]+[q],&\\
A_1 +[3]+[3,2]+[q], &\\ A_1
+[3]+[5]+\frac{1}{q}(1,q_1).\end{array}$$

\bigskip\noindent Thus,
 $$(-F_1^2, -F_2^2, -F_3^2)=(2,7, 13), (2,3,q),
(2,5,q), (3,5,q), (2,3,5).$$
 Then Lemma \ref{dp3} rules out the first four possibilities.

 In the last case $(-F_1^2, -F_2^2, -F_3^2)=(2,3,5)$, $F_i=f^{-1}(p_i)$ for
 $i=1,2,3$. In this case we consider the sublattice $$\langle C, F_1, F_2,
F_3\rangle\subset H^2(S', \mathbb{Z})$$ generated by $C, F_1, F_2,
F_3$. It is of rank 4 and has
\begin{displaymath}
\left( \begin{array}{cccc}
-1 & 1 & 1 & 1 \\
1 & -2 & 0 & 0 \\
1 & 0 & -3 & 0 \\
1 & 0& 0 & -5 \\
\end{array} \right)
\end{displaymath}
as its intersection matrix. It has determinant $-1$, hence the
orthogonal complement of $\langle C, F_1, F_2, F_3\rangle$ in
$H^2(S', \mathbb{Z})$ is unimodular. The orthogonal complement is
an over-lattice of the lattice $R_{p_4}$ generated by the
components of $f^{-1}(p_4)$. Since $R_{p_4}$ is a primitive
sublattice of $H^2(S', \mathbb{Z})$, it must be unimodular, hence
$q=1$, a contradiction.
\end{proof}

\subsection{Step 4.} If one of the cases $(2,3,5,q)$,  $q \geq 7$,
$\gcd(q, 30) = 1$, occurs, then $C.f^{-1}(p_4) = 1$.

\begin{proof} Suppose that the case $(2,3,5,q)$ occurs for some $q \geq 7$ with $\gcd(q, 30) =
1$. By Lemma \ref{noA2}, $p_2$ is of type $[3]$.

By Lemma \ref{dp8}, $C.f^{-1}(p_i) \le 1$ for $i=1,2,3,4$.\\
Suppose on the contrary that $C.f^{-1}(p_4) = 0$.\\ Then,
$$C.f^{-1}(p_1) = C.f^{-1}(p_2) = C.f^{-1}(p_3) =1.$$
 Let $F_i\subset f^{-1}(p_i)$ be the component with $CF_i=1$ for $i=1,2,3$.

  Assume that $p_3$ is of type
$[5]$. Then $(-F_1^2, -F_2^2, -F_3^2)=(2,3,5)$ and the sublattice
$\langle C, F_1, F_2, F_3\rangle\subset H^2(S', \mathbb{Z})$ has
determinant $-1$, leading to the same contradiction as above,
since the orthogonal complement of $\langle C, F_1, F_2,
F_3\rangle$ in $H^2(S', \mathbb{Z})$ is $R_{p_{4}}$.

Assume that $p_3$ is of type $[2, 3]$. Then $(-F_1^2, -F_2^2,
-F_3^2)=(2,3,2)$ or $(2,3,3)$. Let $f^{-1}(p_3) =F_3+ F_3'$. If
$F_3^2=-2$, then $$|\det\langle C, F_1, F_2, F_3,
F_3'\rangle|=13,$$ and by Lemma \ref{dp10}(2-a) $L = 2+2+3+2 = 9$,
so $l = 5$. The orthogonal complement of $\langle C, F_1, F_2,
F_3, F_3'\rangle$ in $H^2(S', \mathbb{Z})$ is $R_{p_{4}}$, hence
$$|\det (R_{p_{4}})|=q=13.$$ This leads to a contradiction since
there is no continued fraction of length $5$ with $q = 13$. If
$F_3^2=-3$, then $$|\det\langle C, F_1, F_2, F_3,
F_3'\rangle|=7,$$ hence $|\det (R_{p_{4}})|=q=7.$ By Lemma
\ref{dp10}(2), $L = 2+2+3+3 = 10$, so $l = 6$. Thus $p_4$ is of
type $A_6$. But, then $$K^2_S = 9-L
-\mathcal{D}_{p_2}^2-\mathcal{D}_{p_3}^2= -1 + \frac{1}{3} +
\frac{2}{5} < 0,$$ a contradiction.

Assume that $p_3$ is of type $A_4=[2,2,2,2]$. Then $(-F_1^2,
-F_2^2, -F_3^2)=(2,3,2)$. Let $f^{-1}(p_3)=H_1+H_2+H_3+H_4$. If
$F_3$ is an end component of $f^{-1}(p_3)$, say $H_1$,  then
$$|\det\langle C, F_1, F_2, H_1,H_2,H_3,H_4\rangle|= 19,$$ and by
Lemma \ref{dp10}(2-a) $L = 2+2+3+2 = 9$, so $l = 3$. Thus $|\det
(R_{p_{4}})|=q=19$ and rank$(R_{p_{4}})=3.$ Among all
Hirzebruch-Jung continued fractions of order 19, only two,
$[7,2,2]$ and $[3,4,2]$, have length 3. In each of these two
cases, $f^{-1}(p_4)$  contains an irreducible component with
self-intersection $\leq -4$. Since
$f^{-1}(p_4)\subset\mathcal{F}-F_1-F_2-F_3$, we have a
contradiction by Lemma \ref{dp10}(2-b). If $F_3$ is a middle
component of $f^{-1}(p_3)$, say $H_2$, then $$|\det\langle C, F_1,
F_2, H_1,H_2,H_3,H_4\rangle|= 31,$$  and by Lemma \ref{dp10}(2-a)
$L = 2+2+3+2 = 9$, so $l = 3$. Thus $q=31$ and $p_4$ is of type
$[11,2,2], [3,6,2],$ or $[5,2,4]$. In each of these three cases,
$f^{-1}(p_4)$ contains an irreducible component with
self-intersection $\leq -4$, a contradiction by Lemma
\ref{dp10}(2-b). This proves that $C.f^{-1}(p_4) = 1.$
\end{proof}

\subsection{Step 5.} None of the cases $(2,3,5,q)$,  $q \geq 7$, $\gcd(q,
30) = 1$, occurs.

\begin{proof} Suppose that the case $(2,3,5,q)$ occurs for some $q \geq 7$ with $\gcd(q, 30) =
1$.\\ By Lemma \ref{noA2}, $p_2$ is of type $[3]$.\\ By Step 2,
$C\mathcal{F}=3$ and $C$ meets the three components $F_1, F_2,
F_3$ of $\mathcal{F}$.\\ By Step 3,
$$2C+F_1 + F_2 +F_3+ K_{S'} \sim \Gamma$$ for some $(-1)$-curve
$\Gamma$.\\ By Step 4, we may assume that $F_3 \subset
f^{-1}(p_4).$

 Let $$f^{-1}(p_4) =
\underset{D_1}{\overset{-n_1}{\circ}}-\underset{D_2}{\overset{-n_2}{\circ}}-\ldots-
\underset{D_l}{\overset{-n_l}{\circ}}$$

\bigskip\noindent
 and $F_3=D_j$ for some $1\le j\le l$. Note first that by Lemma
\ref{dp10}(2-b), $n_k\le 3$ for all $k\neq j$.

\medskip
Assume that $p_3$ is of type $[5]$. By Lemma \ref{dp10}(2-b), $C$
must meet $f^{-1}(p_3)$, so we may assume that $F_2=f^{-1}(p_3)$.
Since $F_1=f^{-1}(p_1)$ or $F_1=f^{-1}(p_2)$, by Lemma \ref{dp3},
 $$(-F_1^2, -F_2^2,
-F_3^2)=(2,5,2), (3,5,2), (2,5,3).$$
 By Lemma \ref{dp10} (2-a), we have
 $$(L,n_j)=(11,2), (12,2), (12,3),$$ hence
$$(l,n_j)=(8,2), (9,2), (9,3).$$
By Lemma \ref{dp10} (2-b) and (2-c), $$\begin{array}{ll}
[n_1,\ldots,
n_l]=&[3,2,2,2,2,2,2,2], [2,2,2,2,2,2,2,2];\\
&[3,2,2,2,2,2,2,2,2], [2,2,2,2,2,2,2,2,2]
\end{array}$$ up to permutation of
$n_1,\ldots, n_l$.
Counting all possible permutations and identifying $[n_1,\ldots,
n_l]$ with its reverse $[n_l,\ldots, n_1]$, it is easy to see that
there are $$4+1+5+1=11$$ possible cases for $[n_1,\ldots,
n_l]$.\\ E.g., $[3,2,2,2,2,2,2,2]$ gives 4 possible cases for $[n_1,\ldots,
n_l]$. None of these 11 cases satisfies the following three
conditions:

\bigskip
\begin{itemize}
\item ($\#$1) $K^2_S > 0$,
\item ($\#$2) $\gcd(q, 30) = 1$,
\item ($\#$3) $D=|\det(R)|K^2_S$ is a positive square integer.
\end{itemize}

\bigskip Assume that $p_3$ is of type $[2,3]$. Then, by Lemma
\ref{dp3},
$$(-F_1^2, -F_2^2, -F_3^2)=(2,3,n_j), n_j\le 5,\,\,\, {\rm or}\,\,
(3,3,n_j), n_j=2,\,\,\, {\rm or}\,\,(2,2,n_j).$$ The last case can
be ruled out by Lemma \ref{dp12} and Step 1 since $S$ has only cyclic singularities. Now, by Lemma \ref{dp10}(2), we have
$$(l,n_j)=(5,2), (6,3), (7,4), (8,5), (6,2),$$ and  $$\begin{array}{lll} [n_1,\ldots,
n_l]&=&[3,2,2,2,2], [2,2,2,2,2]; [3,2,2,2,2,2];\\ &&
[4,2,2,2,2,2,2];
 [5,2,2,2,2,2,2,2]; [2,2,2,2,2,2],\end{array}$$ up to permutation of $n_1,\ldots,
n_l$. It is easy to see that there
are $16$ possible cases for $[n_1,\ldots, n_l]$. None of
them satisfies the three conditions $(\#1),(\#2),(\#3)$.

\medskip
Assume that $p_3$ is of type $[2,2,2,2]$. Then, by Lemma
\ref{dp3}, $$(-F_1^2, -F_2^2, -F_3^2)=(2,3,n_j), n_j\le 5,
\,\,\,{\rm or}\,\,\,(2,2,n_j).$$ The last case can be ruled out by
Lemma \ref{dp12} and Step 1 since $S$ has only cyclic
singularities. Now, by Lemma \ref{dp10}(2), we have
$$(l,n_j)=(3,2), (4,3), (5,4), (6,5),$$ and
$$\begin{array}{lll} [n_1,\ldots, n_l]&=&[3,2,2],
[2,2,2]; [3,2,2,2]; [4,2,2,2,2]; [5,2,2,2,2,2],
\end{array}$$ up to permutation of $n_1,\ldots, n_l$. It is easy to see that there are
$11$ possible cases for $[n_1,\ldots, n_l]$.  None of them
satisfies the three conditions $(\#1),(\#2),(\#3)$.
\end{proof}

\medskip
Next, we will show that none of the cases $(2,3,7,q)$, $11 \leq q
\leq 41$, $\gcd(q, 42) = 1$, and $(2,3,11,13)$  occurs. To do
this, it is enough to consider the 24 cases of Table 1.

\subsection{Step 6.} None of the  $24$ cases of Table 1  occurs.

\begin{proof} By Step 2, $C\mathcal{F}=3$ in each of the 24 cases of Table 1.\\

Each of Cases (1), (2), (3), (4), (6), (8), (9), (11), (12), (13),
(17), and (19), contains an irreducible components $F'$ with
self-intersection $\le -6$. Lemma \ref{dp10} (2-b) implies that
$C$ meets $F'$. Thus $C$ meets two components of $\mathcal{F}$
with self-intersection $-2$ by Lemma \ref{dp3}. Thus we get a
contradiction for those cases by Lemma \ref{dp12} and Step 1.

By Lemma \ref{dp10} (2-c), we get a contradiction immediately for
Cases (7), (10), (14), (16), (18), since each of these cases
contains a connected component of $\mathcal{F}$ with at least two
irreducible components of self-intersection $\le -3$.

By Lemma \ref{dp3} and Lemma  \ref{dp10} (2-b), we get a
contradiction immediately for Cases (5), (20), (21), (22), since
each of these cases contains at least two irreducible components
with self-intersection $\le -4$.

We need to rule out the remaining three cases: (15), (23), (24).

\medskip
Consider Case (24). Note that $L = 10$ in this case. On the other
hand, by Lemma \ref{dp10}(2-b), $C$ must meet the component having
self-intersection number $-5$. Thus, we may assume that
$F_3^2=-5$. Since $F_1^2\le -2, F_2^2\le -2$, Lemma  \ref{dp10}
(2-a) gives $L=2 - (F^2_1 + F^2_2 + F^2_3) \geq 2 + 2+2+ 5 = 11$,
a contradiction.

\medskip
Case (15): Let
 $$\underset{A}{\overset{-2}{\circ}} \quad \underset{B}{\overset{-3}{\circ}} \quad \underset{C_1}{\overset{-3}{\circ}}-\underset{C_2}{\overset{-2}{\circ}}
 -\underset{C_3}{\overset{-2}{\circ}}\quad \underset{D_1}{\overset{-3}{\circ}}-\underset{D_2}{\overset{-2}{\circ}}-
 \underset{D_3}{\overset{-2}{\circ}}-\underset{D_4}{\overset{-2}{\circ}}-
 \underset{D_5}{\overset{-2}{\circ}}$$

 \bigskip\noindent
 be the exceptional curves.
In this case, $K^2_S = \frac{50}{231}, \sqrt{D} = 10$.\\ Since $L = 10=2 - (F^2_1 + F^2_2 + F^2_3)$, $C$ meets only two of $B, C_1, D_1$.\\
 If $CC_1 = CD_1 =1$, then $CA = 1$. Applying Proposition \ref{int}(1) to $C$ of the form \eqref{C} and
looking at Table \ref{dp-315}, we get
 $$\frac{m}{\sqrt{D}}K^2_S  = 1 - \frac{3}{7} - \frac{5}{11} = \frac{9}{77},$$
 thus $m = \frac{27}{5}$, not an integer, a contradiction.
 \begin{table}[h]
\caption{} \label{dp-315}
\renewcommand\arraystretch{1.5}
\noindent\[
\begin{array}{|c|c|c|c|c|c|c|c|c|c|c|}
\hline
 & [2] & [3] & \multicolumn{3}{|c|}{[3,2,2]} & \multicolumn{5}{|c|}{[3,2,2,2,2]}  \\   \hline
j&1&1&1&2&3 &1 &2&3&4&5 \\
 \hline
1-\frac{v_j + u_{j}}{q} & 0 &\frac{1}{3}& \frac{3}{7}&\frac{2}{7}&\frac{1}{7}&\frac{5}{11}&\frac{4}{11}  & \frac{3}{11}& \frac{2}{11}& \frac{1}{11} \\
  \hline
\end{array}
\]
\end{table}\\
If  $CB = CC_1=CA =1$, then $\Gamma$ meets $C_2$ and $D_1$ only, a
contradiction to Lemma \ref{blow2}.\\
If  $CB = CC_1=CD_j =1$ for some $j$, then Proposition
\ref{int}(1) gives
 $$\frac{m}{\sqrt{D}}K^2_S = 1 -\frac{1}{3} - \frac{3}{7} - \Big(1 - \frac{v_j + u_j}{q} \Big) > 0,$$
hence $ j = 4, 5.$ If $j = 4$, then
  $$\frac{m}{\sqrt{D}}K^2_S = 1 -\frac{1}{3} - \frac{3}{7} - \frac{2}{11}=\frac{13}{231},$$
 thus $m = \frac{13}{5}$, a contradiction. If $j = 5$, then
  $$\frac{m}{\sqrt{D}}K^2_S = 1 -\frac{1}{3} - \frac{3}{7} - \frac{1}{11} = \frac{34}{231},$$
 thus $m = \frac{34}{5}$, a contradiction.\\
If $CB = CD_1=CA = 1$, then
  $$\frac{m}{\sqrt{D}}K^2_S = 1 -\frac{1}{3}  - \frac{5}{11} = \frac{7}{33},$$
 thus $m = \frac{49}{5}$, a contradiction.\\ If $CB = CD_1=CC_2 = 1$, then
  $$\frac{m}{\sqrt{D}}K^2_S = 1 -\frac{1}{3}  -\frac{2}{7} - \frac{5}{11} = -\frac{17}{231} < 0,$$
  a contradiction.\\ If $CB = CD_1=CC_3 = 1$, then
  $$\frac{m}{\sqrt{D}}K^2_S = 1 -\frac{1}{3}  -\frac{1}{7} - \frac{5}{11} = \frac{16}{231},$$
 thus $m = \frac{16}{5}$,  a contradiction.

\medskip
 Case (23): Let
 $$\underset{A}{\overset{-2}{\circ}} \quad \underset{B}{\overset{-3}{\circ}} \quad \underset{C_1}{\overset{-3}{\circ}}-\underset{C_2}{\overset{-2}{\circ}}-
 \underset{C_3}{\overset{-2}{\circ}}-\underset{C_4}{\overset{-2}{\circ}}-
 \underset{C_5}{\overset{-2}{\circ}}\quad  \underset{D_1}{\overset{-4}{\circ}}-\underset{D_2}{\overset{-2}{\circ}}-
 \underset{D_3}{\overset{-2}{\circ}}-\underset{D_4}{\overset{-2}{\circ}}$$

 \bigskip\noindent
 be the exceptional curves.
 Since $C$ meets $D_1$ and $L = 11$, $C$ must meet only one of $B$ and
 $C_1$.\\
  If $CB=CA =1$, then $\Gamma$ meets exactly two irreducible components $C_1, D_2$ with multiplicity 1, a contradiction to Lemma
  \ref{blow2}.\\
 If $CB=CC_j = 1$ for some $j\ge 2$, then Table \ref{dp-323} gives
  $$\frac{m}{\sqrt{D}}K^2_S \le 1 - \frac{1}{3} -\frac{1}{11} - \frac{8}{13}   < 0,$$
  a contradiction.\\
  If $CC_1 = 1$, then $CA=1$ and Proposition \ref{int}(1) together with Table \ref{dp-323} gives
  $$\frac{m}{\sqrt{D}}K^2_S = 1 - 0 -\frac{5}{11} - \frac{8}{13}   < 0,$$
  a contradiction.
 \begin{table}[ht]
\caption{} \label{dp-323}
\renewcommand\arraystretch{1.5}
\noindent\[
\begin{array}{|c|c|c|c|c|c|c|c|c|c|c|c|}
\hline
 & [2] & [3] & \multicolumn{5}{|c|}{[3,2,2,2,2]} & \multicolumn{4}{|c|}{[4,2,2,2]}  \\   \hline
j&1&1&1&2&3& 4&5&1 &2&3&4 \\
 \hline
1-\frac{v_j + u_{j}}{q} & 0 &\frac{1}{3}& \frac{5}{11}&\frac{4}{11}&\frac{3}{11}&\frac{2}{11}&\frac{1}{11}  & \frac{8}{13}& \frac{6}{13}& \frac{4}{13}& \frac{2}{13}\\
  \hline
\end{array}
\]
\end{table}\\
\end{proof}

This completes the proof of Theorem \ref{dpmain}.


\begin{thebibliography}{99}
\bibitem[Be]{Belousov} G. N. Belousov, \textit{Del Pezzo surfaces with log terminal singularities}, Math. Notes \textbf{83} (2008), no. 2, 152-161.


\bibitem[Br]{Brieskorn} E. Brieskorn, \textit{Rationale Singularit\"aten komplexer Fl\"achen}, Invent. Math. \textbf{4} (1968), 336-358.


\bibitem[FS85]{FS85} R. Fintushel and R. Stern, \textit{Pseudofree
orbifolds}, Ann. of Math. (2) \textbf{122} (1985), no. 2, 335-364.

\bibitem[FS87]{FS87} R. Fintushel and R. Stern,  \textit{O(2) actions on the
5-sphere}, Invent. Math. \textbf{87} (1987), no. 3, 457-476.

\bibitem[GZ]{GZ} R. V. Gurjar and D. Q. Zhang, \textit{$\pi_1$ of smooth points of a log del Pezzo surface is finite: I}, J. Math. Sci. Univ. Tokyo \textbf{1} (1994), 137-180.

\bibitem[HK1]{HK1} D. Hwang and J. Keum, \textit{The maximum number of singular points on rational homology projective planes}, arXiv:0801.3021, to appear in J. Algebraic Geom.

\bibitem[HK2]{HK2} D. Hwang and J. Keum, \textit{Algebraic Montgomery-Yang Problem: the noncyclic case}, arXiv:0904.2975, to appear in Math. Ann.

\bibitem[HK3]{HK3} D. Hwang and J. Keum, \textit{Construction of singular rational surfaces of Picard number one with ample canonical divisor}, arXiv:1007.1936.

\bibitem[Ke07]{Keum} J. Keum, \textit{A rationality criterion for projective surfaces - partial solution to Kollar's conjecture},
Algebraic geometry, 75-87, Contemp. Math., \textbf{422}, Amer. Math. Soc., Providence, RI, 2007.

\bibitem[Ke08]{K08} J. Keum, \textit{Quotients of Fake Projective Planes}, Geom. \& Top. \textbf{12} (2008), 2497-2515.

\bibitem[Keu10]{K10} J. Keum, \textit{The moduli space of $\mathbb{Q}$-homology projective planes with 5 quotient singular points}, Acta Math. Vietnamica \textbf{35} (2010), 79-89.

\bibitem[KM]{KM} S. Keel and J. McKernan, \textit{Rational curves on quasi-projective surfaces}, Mem. Amer. Math. Soc. \textbf{140} (1999), no. 669.

\bibitem[KNS]{KNS} R. Kobayashi, S. Nakamura, and F. Sakai \textit{A
numerical characterization of ball quotients for normal surfaces
with branch loci}, Proc. Japan Acad. Ser. A,  Math. Sci.
\textbf{65} (1989), no. 7, 238-241.

\bibitem[Kol05]{Kol05} J. Koll\'ar, \textit{Einstein metrics on $5$--dimensional Seifert bundles}, Jour. Geom. Anal.
\textbf{15} (2005), no. 3, 463-495.

\bibitem[Kol08]{Kol08} J. Koll\'ar, \textit{Is there a topological Bogomolov-Miyaoka-Yau inequality?}, Pure Appl. Math. Q. (Fedor Bogomolov special issue, part I), \textbf{4} (2008), no. 2,  203--236.

\bibitem[LW]{LW} E. Looijenga, and J. Wahl, \textit{Quadratic functions and smoothing surface singularities},
Topology \textbf{25} (1986), no.3, 261-291.



\bibitem[Me]{Megyesi} G. Megyesi, \textit{Generalisation of the Bogomolov-Miyaoka-Yau inequality to singular surfaces},
Proc. London Math. Soc. (3) \textbf{78} (1999), 241-282.


\bibitem[Mi]{Miyaoka} Y. Miyaoka, \textit{The maximal number of quotient singularities on surfaces with given numerical invariants},
Math. Ann. \textbf{268} (1984), 159-171.

\bibitem[MY]{MY} D. Montgomery and C. T. Yang, \textit{Differentiable pseudo-free circle
actions on homotopy seven spheres}, Proc. of the Second Conference
on Compact Transformation Groups (Univ. Massachusetts, Amherst,
Mass., 1971), Part I, Springer, 1972, pp. 41-101. Lecture Notes in
Math., Vol. 298.

\bibitem[P]{Pet} T. Petrie, \textit{Equivariant quasi-equivalence, transversality, and normal cobordism}, Proc. Int. Cong. Math., Vancouver, 1974, 537-541.



\bibitem[S]{Sakai} F. Sakai, \textit{Semistable curves on algebraic surfaces and logarithmic pluricanonical maps}, Math. Ann. \textbf{254} (1980), no. 2, 89-120.

\bibitem[Se]{Sei} H. Seifert, \textit{Topologie dreidimensionaler gefaserter R\"aume}, Acta Math. \textbf{60} (1932), 147-238.

\bibitem[Z]{Zhang} D. Q. Zhang, \textit{Logarithmic del Pezzo surfaces of rank one with contractible boundaries}, Osaka J. Math. \textbf{25} (1988), 461-497.


\end{thebibliography}
\end{document}